\documentclass[twoside]{amsart}
\usepackage{mathrsfs}
\usepackage{amsfonts}
\usepackage{amssymb}
\usepackage{amssymb}
\usepackage{amssymb}
\usepackage{amssymb}
\usepackage{amssymb}
\usepackage{amssymb}
\usepackage{amssymb}
\usepackage{amssymb}
\usepackage{amsmath}
\usepackage{amsthm}
\usepackage[all]{xy}

\usepackage{lastpage}

\RequirePackage{amsmath} \RequirePackage{amssymb}
\usepackage{amscd,latexsym,amsthm,amsfonts,amssymb,amsmath,amsxtra}
\usepackage[colorlinks=true, urlcolor=blue,bookmarks=true,bookmarksopen=true,
citecolor=blue,hypertex]{hyperref}

    \newcommand{\BC}{{\mathbb {C}}}

     \newcommand{\bW}{{\bf {W}}}

    \newcommand{\bx}{{\bf {x}}}   \newcommand{\bm}{{\bf {m}}}   \newcommand{\bn}{{\bf {n}}} \newcommand{\bt}{{\bf {t}}}

    \newcommand{\CO}{{\mathcal {O}}} \newcommand{\CP}{{\mathcal {P}}}
     
    \newcommand{\CS}{{\mathcal {S}}} 
     
    \newcommand{\CW}{{\mathcal {W}}}

       \newcommand{\fD}{{\mathfrak{D}}}

    \newcommand{\RG}{{\mathrm {G}}} \newcommand{\RH}{{\mathrm {H}}}

    \newcommand{\RU}{{\mathrm {U}}}

    \newcommand{\diag}{{\mathrm {diag}}}
     
      \newcommand{\Mat}{{\mathrm {Mat}}}

     \newcommand{\Nm}{{\mathrm {Nm}}}
    \newcommand{\lenth}{{\mathrm {\lenth}}}

     \newcommand{\GL}{{\mathrm{GL}}}
    \newcommand{\Hom}{{\mathrm{Hom}}} 
    \newcommand{\Ind}{{\mathrm{Ind}}} \newcommand{\ind}{{\mathrm{ind}}}

    \renewcommand{\Re}{{\mathrm{Re}}} 
    \newcommand{\Res}{{\mathrm{Res}}}

 \newcommand{\Vol}{{\mathrm{Vol}}}
\renewcommand{\mod}{\ \mathrm{mod}\ }

\newcommand{\supp}{\mathrm{supp}}

 \newcommand{\SO}{{\mathrm{SO}}}
 
 \newcommand{\tr}{{\mathrm{tr}}}
  
\newcommand{\vol}{{\mathrm{vol}}}  
 
\newcommand{\Char}{{\mathrm{Char}}}

 \newcommand{\Sp}{{\mathrm{Sp}}} \newcommand{\Herm}{{\mathrm{Herm}}}

    \newcommand{\pair}[1]{\langle {#1} \rangle}
    \newcommand{\wpair}[1]{\left\{{#1}\right\}}

    \newcommand{\incl}{\hookrightarrow}
    
     \newcommand{\ra}{\rightarrow}

    \theoremstyle{plain}

    \newtheorem*{theorem*}{Theorem}

    \newtheorem{thm}{Theorem}[section] \newtheorem{cor}[thm]{Corollary}
    \newtheorem{lem}[thm]{Lemma}  \newtheorem{prop}[thm]{Proposition}
    \newtheorem {conj}[thm]{Conjecture} 
     \newtheorem{rem}[thm]{Remark}

    \numberwithin{equation}{section}

\usepackage[top=1in, bottom=1in, left=1.25in, right=1.25in]{geometry}

\title{A local converse theorem for $\RU(2,2)$}
\author{Qing Zhang}

\address{Department of Mathematics, The Ohio State University,
100 Math Tower 231 West 18th Ave,
Columbus OH 43210}
\email{zhang.1649@osu.edu}

\begin{document}

\maketitle

\begin{abstract}
Let $F$ be a $p$-adic field and $E/F$ be a quadratic extension. In this paper, we prove the local converse theorem for generic representations of $\RU_{E/F}(2,2)$ if $E/F$ is unramified or the residue characteristic of $F$ is odd. Our method is purely local and analytic, and the same method also gives the local converse theorem for $\Sp_4(F)$ and $\widetilde \Sp_4(F)$ if the residue characteristic of $F$ is odd.
\end{abstract}

\setcounter{tocdepth}{1}

\section*{Introduction}

Let $F$ be a $p$-adic field and $E/F$ be a quadratic extension. Let $G=\RU_{E/F}(n,n)$, and $\psi$ be a generic character of its maximal unipotent subgroup. Let $\pi$ be a $\psi$-generic irreducible smooth representation of $G$ and $\tau$ be a generic irreducible representation of $\GL_m(E)$, then one can define a gamma factor $\gamma(s,\pi\times \tau,\psi)$. The local converse problem asks if one can determine the representation $\pi$ uniquely if one knows enough information of the gamma factors $\gamma(s,\pi\times \tau,\psi)$ for various twists. More precisely, we have the following conjecture
\begin{conj}[Local converse conjecture, see Conjecture 6.3 of \cite{JngN}] \label{conj1}
Let $\pi, \pi'$ be two $\psi$-generic irreducible admissible representations of $G(F)$ with the same central character. If $\gamma(s,\pi\times \tau, \psi)=\gamma(s,\pi'\times \tau,\psi)$ for all irreducible supercuspidal representations $\tau$ of $\GL_k(E)$ with $k\le n$, then $\pi_1\cong \pi_2.$
\end{conj}
\noindent \textbf{Remark:} Conjecture 6.3 of \cite{JngN} is stated for all classical groups. Here for simplicity, we only consider the case $G=\RU_{E/F}(n,n)$. On the other hand, there is no central character assumption in Conjecture 6.3 of \cite{JngN}. In the $\GL_n$ case, the equality of the central character is in fact a consequence of the equality of the gamma factors twisting by $\GL_1$, see Corollary 2.7 of \cite{JngNS}. Thus it's natural to expect that this is also true in the classical group case.\\

In this paper, we confirm the above conjecture in the case when $n=2$ and $E/F$ is unramified or $E/F$ is ramified but the residue characteristic of $F$ is not 2. More precisely, we have the following
\begin{theorem*}[Local Converse Theorem for $\RU(2,2)$, Theorem \ref{thm56}]
Suppose that $E/F$ is unramified or $E/F$ is ramified but the residue characteristic of $F$ is odd. Let $\pi,\pi'$ be two $\psi$-generic irreducible admissible representations of $\RU_{E/F}(2,2)(F)$. If $\gamma(s,\pi\times \eta, \psi)=\gamma(s,\pi'\times \eta, \psi)$ and $\gamma(s,\pi\times \tau, \psi)=\gamma(s,\pi'\times \tau, \psi)$ for all quasi-characters $\eta$ of $E^\times$ and all irreducible admissible representations $\tau$ of $\GL_2(E)$, then $\pi\cong \pi'$.
\end{theorem*}
The $\gamma$-factors used here are of Rankin-Selberg type, which are defined using Shimura type integrals. In \cite{Ka}, the Rankin-Selberg $\gamma$-factors for the group $\Sp_{2n}$ are studied. In particular, it is proved in \cite{Ka} that the Rankin-Selberg gamma factors are the same as the local gamma factors defined using Langlands-Shahidi method. Unfortunately, the unitary group case is not included in \cite{Ka}. Since the local zeta integrals used in the unitary case to define the gamma factors are totally parallel to the symplectic case, it is natural to believe that the gamma factors used here are multiplicative and are the same as the Langlands-Shahidi local gamma factors up to a normalizing factor. Once we know the gamma factors are multiplicative, it suffices to twist supercuspidal representations of $\GL_2(E)$ in the above local converse theorem.

We also mention that, based on the same methods, we can also prove the local converse theorems for $ \Sp_4(F)$ and $\widetilde \Sp_4(F)$ when $F$ is a $p$-adic field with odd residue characteristic.

Our proof of the local converse theorem is based on detailed analysis on the partial Bessel functions associated with Howe vectors. In \cite{Ba1, Ba2}, E. M. Baruch proved local converse theorem for $\RG\Sp_4$ and $\RU(2,1)$ using Howe vectors. In \cite{Zh2}, we proved the stability of the Rankin-Selberg gamma factors for $\Sp_{2n}$ and $\widetilde \Sp_{2n}$ using Baruch's methods, and remarked that this method might be used to prove the local converse theorem for $\Sp_{2n}$, $\widetilde \Sp_{2n}$ and $\RU(n,n)$ once we can extend a stability property of partial Bessel functions associated with Howe vectors (Theorem 3.11 of \cite{Zh2}) to the most general case. In this paper, we illustrate how to get such a local converse theorem in the small ranked case. In fact,  in the case $n=2$, the stability property of partial Bessel functions associated with Howe vectors can be checked directly because the Weyl group of $\RU(2,2)$ is small, see Proposition \ref{prop25}. We expect our method can be used to give local converse theorems for more general groups. \\

In this paper, we also construct a new local gamma factor $\gamma'(s,\pi\times \eta,\psi)$ for a generic representation $\pi$ of $\RU_{E/F}(2,2)$ and a quasi-character $\eta$ of $E^\times$. This new gamma factor is defined by Hecke type local zeta integrals, which are easier to handle than the Shimura type integrals. This construction can be extended to the case $\RU(n,n)\times \GL_m$ when $m<n$. But it is not known whether this new local zeta integrals come from global zeta integrals.\\

This paper is organized as follows. In $\S$1, we review the definition of $\gamma$-factors for $\RU_{E/F}(2,2)\times \GL_k(E)$ with $k\le 2$ after \cite{Ka}. In $\S$2, we review the definition of Howe vectors and a stability property of Howe vectors. We constructed some sections of induced representations in $\S$3 which will be used in the later calculation. In $\S$4 and $\S$5,  we consider the gamma factors twisting by $\GL_1$ and $\GL_2$ and finish the proof of the local converse theorem when $E/F$ is unramified. In $\S$6, we construct a new gamma factor $\gamma'(s,\pi\times \eta,\psi)$ for a generic representation $\pi$ of $\RU(2,2)$ and a quasi-character $\eta$ of $E^\times$. We also show that this new gamma factor can replace the old gamma factor in the local converse theorem. In $\S$7, we give a brief account of the proof of the local converse theorem in the case $E/F$ is ramified and the residue characteristic of $F$ is odd. In $\S$7, we consider the local converse theorem for $\Sp_4(F)$ and $\widetilde \Sp_4(F)$ when $F$ is a local field with odd residue characteristic. We also explain that the local converse theorem for $\Sp_4(F)$ is in fact true without the assumption on $F$ and the central character, based on the local Langlands correspondence for $\Sp_4$, \cite{GT2}.

\section*{Acknowledgements}
I would like to thank my advisor Professor James W. Cogdell for his constant encouragement, generous support and countless hours he spent on this work.  The influence of E. M. Baruch's thesis \cite{Ba1} on this paper should be evident for the reader. I would like to express my appreciation to Professor Baruch for his pioneer work.

\section*{Notations} Let $E/F$ be a quadratic extension of local fields and let $\epsilon_{E/F}$ be the local class field theory character of $F^\times$. Denote the nontrivial Galois action by $x\mapsto \bar x$ for $x\in E$. Let $\CO_E$ (resp. $\CO_F$) be the ring of integers of $E$ (resp. $F$), and let $\CP_E$ (resp. $\CP_F$) be the maximal idel of $\CO_E$ (resp. $\CO_F$). Let $q_E= |\CO_E/\CP_E|$ and $q_F=|\CO_F/\CP_F|$. Let $E^1=\wpair{x\in E^\times: x\bar x=1}$.

\subsection*{$\RU(2,2)$ and its subgroups}
Let $G=\RU_{E/F}(2,2)$ be the isometry group of the Hermitian form defined by $$s=\begin{pmatrix} &1_2\\ -1_2 &\end{pmatrix},$$
where $1_2$ is the $2\times 2$ identity matrix. Explicitly, $$G(F)=\wpair{g\in \GL_{4}(E):  g s {}^t \bar g=s.}$$ 
When the field extension $E/F$ is understood, we will ignore it from the notation, and write $\RU(2,2)$ instead of $\RU_{E/F}(2,2)$. Let $P=MN$ be the Siegel parabolic subgroups with Levi subgroup
$$M(F)=\wpair{\bm(a):=\begin{pmatrix}a& \\ & {}^t \bar a^{-1} \end{pmatrix}, a\in \GL_2(E)},$$
and unipotent subgroup $$N(F)=\wpair{\bn(b):=\begin{pmatrix}1&b\\ &1 \end{pmatrix}, b\in \Herm_2(F)},$$
where $\Herm_2(F)=\wpair{x\in \textrm{Mat}_{2\times 2}(E): x={}^t\bar x}.$
Let $T$ be the maximal torus in $M$. A typical element of $T$ is of the form $\bt(a_1,a_2)=\diag(a_1,a_2,\bar a_1^{-1},\bar a_2^{-1})\in T$ with $a_1,a_2\in E^\times$. We also use the  notation
$$\bt(a):=\bt(a,1),a\in E^\times.$$

Let $U$ be the maximal unipotent subgroup defined by 
$$U=\wpair{\bm(u)n: u=\begin{pmatrix}1& x\\ &1 \end{pmatrix}\in \GL_2(E), n\in N}.$$
Let $Z$ be the center of $G$. Then $Z=\wpair{\bt(z,z),z\in E^1}\cong E^1$. Denote $B=TU$, a Borel subgroup of $G$. 

From the isomorphism $M\cong \GL_2(E)$, we view $\GL_2(E)$ as a subgroup of $G$. In $\GL_2(E)$, we denote 
$B^{(2)}=T^{(2)}N^{(2)}$ the upper triangular Borel subgroup with $T^{(2)}$ the torus and $N^{(2)}$ be the upper triangular unipotent subgroup. 
\subsection*{Roots and Weyl group}
The group $G$ has two simple roots defined by 
$$\alpha(\bt(a_1,a_2))=a_1/a_2, \beta(\bt(a_1,a_2))=a_2\bar a_2, a_1,a_2\in E^\times.$$
The positive roots are $\Sigma^+=\wpair{\alpha,\beta,\alpha+\beta, 2\alpha+\beta}$. Let $s_\alpha$ be the simple reflection of $\alpha$ and $s_\beta$ be the simple reflection of $\beta$. The Weyl group $\textbf{W}$ of $G$ is 
$$\wpair{1,s_\alpha,s_\beta,s_\alpha s_\beta, s_\beta s_\alpha, s_\alpha s_\beta s_\alpha, s_\beta s_\alpha s_\beta, s_\alpha s_\beta s_\alpha s_\beta}.$$
We denote 
$$w_0=(s_\alpha s_\beta)^2, w_1=s_\beta s_\alpha s_\beta, w_2=s_\alpha s_\beta s_\alpha.$$

For $w\in \bW$, we will fix a representative $\dot w\in G$ of $w$ by 

$$\dot s_\alpha=\begin{pmatrix}&1&&\\ 1& &&\\ &&&1\\ &&1& \end{pmatrix}, \dot s_\beta=\begin{pmatrix}1&&&\\ &&&1\\ &&1&\\ &-1&& \end{pmatrix}$$
and $(ww')^{\dot~}=\dot w \dot w'$. Then 
$$\dot w_0=\begin{pmatrix}&I_2\\ -I_2& \end{pmatrix}$$
We have the relation
$$s_\alpha(\beta)=2\alpha+\beta, s_\beta(\alpha)=\alpha+\beta.$$
For a root $\gamma,$ let $U_\gamma $ be the one parameter subgroup associated to $\gamma$. Let $\bx_\gamma: F\ra U_\gamma$, or $\bx_\gamma:E\ra U_\gamma$ be the corresponding isomorphism. For example, 
$$U_\alpha=\wpair{\bx_\alpha(r)=\begin{pmatrix}1& r&&\\ &1&&\\ &&1&\\ &&-\bar r &1 \end{pmatrix},r\in E}, U_{-\alpha-\beta}=\wpair{\bx_{-\alpha-\beta}(r)=\begin{pmatrix}1& &&\\ &1&&\\ &\bar r&1&\\r &&&1 \end{pmatrix},r\in E}.$$

\section{Local gamma factors for $\RU(2,2)\times \Res_{E/F}(\GL_k), k=1,2$}
\subsection{Weil representations of $\RU(1,1)$}
Let $W$ be the 2-dimensional skew-Hermitian space with skew-Hermitian structure defined by 
$$(w_1,w_2)_W= w_1 \begin{pmatrix} &1\\ -1&\end{pmatrix} {}^t \bar w_2,$$
where $w_1,w_2\in W$ are viewed as row vectors.
Let $G_1=\RU(1,1)=\RU(W)$ be the isometry group of $W$, i.e.,
$$G_1(F)=\wpair{g\in \GL_2(E): g\begin{pmatrix} &1\\ -1& \end{pmatrix}{}^t \bar g= \begin{pmatrix} &1\\ -1& \end{pmatrix} }.$$
In the group $G_1$, we will use the following notations:
$$\bm_1(a)=\begin{pmatrix}a & \\ & \bar a^{-1} \end{pmatrix}, \textrm{ for } a\in E^\times, \bn_1(x)=\begin{pmatrix}1& x\\ &1 \end{pmatrix}, \textrm{ for }x\in F.$$
Let $M_1$ be the subgroup of $G_1$ consisting elements of the form $\bm_1(a),a\in E^\times$ and $N_1$ be the subgroup of $G_1$ consisting elements of the form $\bn_1(x),x\in F$. Let $B_1=M_1N_1$, which is a Borel subgroup of $G_1$.

The skew-Hermitian space $W_{/E}$ can be viewed as a symplectic space over $F$, with the symplectic form $\pair{~,~}_W$ defined by $\tr_{E/F}((~,~)_W)$. We then have an embedding $G_1\ra \Sp(W)$. Let $\mu$ be a character of $E^\times$ such that $\mu|_{F^\times}=\epsilon_{E/F}$, then there is a splitting $s_\mu: G_1\ra \widetilde \Sp(W)$, see \cite{HKS}.

Let $\RH(W)$ be the Heisenberg group associated to $W$. Explicitly, $\RH(W)=W\oplus F$ with the product (written additively) defined by
$$(w_1,t_1)+(w_2,t_2)=(w_1+w_2, t_1+t_2+\frac{1}{2}\pair{w_1,w_2}_W), w_1, w_2\in W, t_1, t_2\in F.$$

Let $J_W=\RU(W)\ltimes \RH(W)$, which is called the Fourier-Jacobi group associated to $W$ in the literature, see \cite{GGP} for example. Let $\psi$ be an additive character of $F$, then we have a Weil representation $\omega_{\psi}$ of the group $\widetilde \Sp(W)\ltimes \RH(W)$, which can be realized on the space $\CS(E)$. For a character $\mu$ of $E^\times$ such that $\mu|_{F^\times}=\epsilon_{E/F}$, we then have a Weil representation $\omega_{\mu,\psi}$ of $J_W$ by the embedding $J_W\ra \widetilde \Sp(W)\ltimes \RH(W)$ induced by $s_\mu$.

We will view $G_1$ as a subgroup of $G=\RU(2,2)$ by the embedding 
$$g=\begin{pmatrix}a& b\\ c&d \end{pmatrix}\mapsto = \begin{pmatrix} 1&&&\\ &a&&b\\ &&1&\\ &c&&d\end{pmatrix}.$$
Let $H$ be the subgroup of $G$ consisting elements of the form
$$[x,y,z]=\bm\begin{pmatrix} 1& x\\ &1 \end{pmatrix}\bn \begin{pmatrix}z& y\\ \bar y&  \end{pmatrix}, x,y\in E,z\in F.$$
We can check that the map
$$\RH(W)\ra H$$
$$((x,y),t)\mapsto [x,y,t-\frac{1}{2}\tr(\bar x y)]$$
defines an isomorphism, and $G_1\cdot H \cong J_W=\RU(W)\ltimes \RH(W)$. 

Thus we get a Weil representation $\omega_{\mu,\psi}$ of $G_1\cdot H$ on $\CS(E)$. The following formulas hold:
\begin{align}
\omega_{\mu,\psi}([0,y,z][x,0,0])\phi(\xi)&=\psi(z+\tr(\xi \bar y))\phi(x+\xi), \label{eq11}\\
\omega_{\mu,\psi}(\bm_1(a))\phi(\xi)&=\mu(a)|a|^{1/2}\phi(\xi a), \nonumber\\
\omega_{\mu,\psi}(\bn_1(b))\phi(\xi)&=\psi(b\xi \bar \xi)\phi(\xi), \nonumber \\
\omega_{\mu,\psi}(w^1)\phi(\xi)&=\epsilon_\psi \int_{E}\psi(-\tr(\bar \xi y))\phi(y)dy. \nonumber
\end{align}
for $\phi\in \CS(E), \xi\in E,$ where $\epsilon_\psi$ is certain Weil index (we don't need its precise definition here), $w^1=\begin{pmatrix}&1\\ -1& \end{pmatrix}$ is the unique nontrivial Weyl element of $G_1$ and $dy$ is the Haar measure on $E$ such that this Fourier transform is self-dual.

\subsection{Weil representations of $\RU(2,2)$}
Let $\psi$ be a nontrivial additive character of $F$, $\mu$ be a character of $E^\times$ such that $\mu|_{F^\times}=\epsilon_{E/F}$ as above. Let $\RU(1)$ be the isometry group of the 1-dimensional Hermitian space $E$ with the Hermitian form $(x,y)=\bar x y$. Then we have a Weil representation $\omega_{\mu,\psi}$ of the pair $G\times \RU(1)$ on $\CS(E^2)$, where $E^2=E\oplus E$. We have the familiar formulas:
\begin{align}
\omega_{\mu,\psi}(h)\Phi(x)&=\Phi(h^{-1}x),\Phi\in \CS(E^2), h\in \RU(1), x\in E^2,\\
\omega_{\mu,\psi}(\bm(a))\Phi(x)&=\mu(\det(a))|\det(a)|^{1/2}\Phi(xa),a\in \GL_2(E), \nonumber\\
\omega_{\mu,\psi}(\bn(b))\Phi(x)&=\psi(x b {}^t\! \bar x)\Phi(x), b\in \Herm_2(F), \nonumber\\
\omega_{\mu,\psi}(w_0)\Phi(x)&= \gamma_\psi \int_{E^2} \Phi(y)\psi(-\tr_{E/F}(x {}^t \!\bar y))dy,      \nonumber
\end{align}
where $x=(x_1,x_2)\in E^2$ is viewed as a row vector and $\gamma_\psi$ is another Weil index.

As a representation of $G=\RU(2,2)$, $\omega_{\mu,\psi}$ is not irreducible. Let $\chi$ be a character of $E^1$, let $\CS(E^2,\chi)$ be the subspace of $\CS(E^2)$ such that $\Phi(zx)=\chi(z)\Phi(x)$ for all $z\in \RU(1)$ and $x\in E^2$. Then $\CS(E^2,\chi)$ is invariant under the action of $G$. Denote this representation by $\omega_{\mu,\psi,\chi}$. Then $\omega_{\mu,\psi,\chi}$ is an irreducible representation of $G$, and 
$$\omega_{\mu,\psi}=\oplus_{\chi\in \hat E^1}\omega_{\mu,\psi,\chi},$$
where $\hat E^1$ denotes the dual group of $E^1$.

\subsection{Induced representation and intertwining operator}
For a quasi-character $\eta$ of $E^\times$, and a complex number $s\in \BC$, let $\eta_s$ be the character of $E^\times $ defined by $\eta_s(a)=\eta(a)|a|_E^s$. We consider the (normalized) induced representation $\Ind_{B_1}^{G_1}(\eta_{s-1/2})$. By definition $\Ind_{B_1}^{G_1}(\eta_{s-1/2})$ consists smooth complex valued functions $f_s$ on $G_1$ such that
$$f_s(\bn_1(b)\bm_1(a)g)=\eta_s(a)f_s(g), b\in F, a\in E^\times, g\in G_1.$$
There is an intertwining operator $M(s): \Ind_{B_1}^{G_1}(\eta_{s-1/2})\ra \Ind_{B_1}^{G_1}(\eta^*_{1/2-s})$ defined by 
$$M(s)(f_s)(g)=\int_{N_1}f_s(w^1ng)dn,$$
where $\eta^*=w^1\eta=\bar \eta^{-1}$. It is well-known that the intertwining operator $M(s)$ is well-defined for $\Re(s)>>0$ and can be meromorphically continued to all $s\in \BC$.

Let $(\tau,V_\tau)$ be an irreducible smooth representation of $\GL_2(E)$, we consider the induced representation $I(s,\tau)=\Ind_{P}^G(\tau|\det|_E^{s-1/2})$. We now fix a nontrivial additive character $\psi_E$ of $E$ such that $\psi_E|_F=1$. For example, we can take a nontrivial additive character $\psi$ of $F$ and a pure imaginary element $e\in E$, and then define $\psi_E(x)=\psi(\tr(ex))$. Given such a character $\psi_E$ of $E$, which is also viewed as a character of the unipotent subgroup $N^{(2)}$ of $\GL_2(E)$, we fix a Whittaker functional $\lambda\in \Hom_{N^{(2)}}(\tau,\psi_E)$. For a section $\xi_s: G\ra V_\tau$, we consider the $\BC$-valued function $f_{\xi_s}(g)$ on $G$ defined by 
$$f_{\xi_s}(g)=\lambda(\xi_s(g)).$$ 
Note that, $\xi_s$ satisfies the relation
\begin{equation}\xi_s(n\bm(a)g)=|\det(a)|_E^{s+1/2}\tau(a)\xi_s(g),n\in N, a\in \GL_2(E).\end{equation}
Thus we get
 \begin{equation}f_{\xi_s}(n\bm(n_1)g)=\psi_E(n_1)f_{\xi_s}(g), n\in N, n_1\in N^{(2)}.\end{equation}
 
 Recall that $w_1=s_\beta s_\alpha s_\beta$, which is the unique element in $\textbf{W}$ such that $w_1(\alpha)$ is positive and simple and $w_1(\beta)<0$. We view $\tau$ as a representation of $M$ and define a representation $\tau^*:=\dot w_1\tau$ of $M$ on the same space $V_\tau$ by conjugation of $w_1$, i.e., $$(\dot w_1 \tau)(\bm(a))=\tau(\dot w_1^{-1}\bm(a)\dot w_1)=\tau(\bm(J_2{}^t\bar a^{-1}J_2)),$$ where $J_2=\begin{pmatrix} &1\\ 1&\end{pmatrix}$. Note that for $n_1=\begin{pmatrix}1&x\\ &1 \end{pmatrix}\in N^{(2)}$, we have 
$$\tau^* (n_1)=\tau\left( \begin{pmatrix}1& -\bar x \\ &1 \end{pmatrix}\right)$$ and by our choice of $\psi_E$, we have $\psi_E(x)=\psi_E(-\bar x)$. Thus the fixed Whittaker functional $\lambda: V_\tau\ra \BC_{\psi_E}$ for $\tau$ gives a $\psi_E$ Whittaker functional for $\tau^*$, i.e., $\lambda\in \Hom_{N^{(2)}}(\tau^*, \psi_E)$.

We consider the normalized induced representation $I(1-s, \tau^*):=\Ind_P^G(\tau^* |\det |^{-s+1/2}).$ For a section $\tilde\xi_{1-s}\in I(1-s,\tau^*)$, we can define $f_{\tilde\xi_{1-s}}$ similar as above and we also have 
\begin{equation}f_{\tilde \xi_{1-s}}(n\bm(n_1)g)=\psi_E(n_1)f_{\tilde \xi_{1-s}}(g).\end{equation}
Consider the standard intertwining operator $$M(s): I(s,\tau)\ra I(1-s, \tau^*)$$
$$M(s)(\xi_s)(g)=\int_{N}\xi_s(n^{-1}\dot w_1g)dn,$$
which is absolutely convergent for $\Re(s)>>0$ and can be meromorphically continued to all $s\in \BC$.
\subsection{The local zeta integral and gamma factor}
For $g\in G$, we denote $j(g)=s_\alpha g s_\alpha$. We fix a nontrivial additive character $\psi$ (resp. $\psi_E$) of $F$ (resp. $E$) such that $\psi_E|_F=1$ and define a character $\psi_U$ of $U$ by 
$$\psi_U\left(\bm\begin{pmatrix}1& x\\ &1 \end{pmatrix}\bn \begin{pmatrix}z & y\\ \bar y & b \end{pmatrix} \right)=\psi_E(-x)\psi(b).$$

Let $\pi$ be an irreducible smooth $\psi_U$-generic representation of $G$, and let $\CW(\pi,\psi_U)$ be the space of $\psi_U$ Whittaker functions on $\pi$. Let $\eta$ be a quasi-character of $E^\times,$ and $\tau$ be an irreducible smooth representation of $\GL_2(E)$.

For $W\in \CW(\pi, \psi_U), \phi_1\in \CS(E),\phi_2\in \CS(E^2), f_s\in I(s,\eta), \xi_s\in I(s,\tau)$, we consider the following integrals 
$$\Psi(W,\phi_1,f_s)=\int_{N_1\setminus G_1}\int_E W( j(\bx_\alpha(x)g))(\omega_{\mu,\psi^{-1}}(g)\phi_1)(x)dx f_s(g) dg,$$ and
$$\Psi(W,\phi_2, f_s)=\int_{U\setminus G} W(g)\omega_{\mu,\psi^{-1}}(g)\phi_2(e_2) f_{\xi_s}(g)dg,$$
where $e_2$ is the row vector $(0,1)\in E^2$ and $f_{\xi_s}$ is defined in the previous section. It is easy to see the above integrals are well-defined formally. Using a standard estimate of the Whittaker function $W$, one can show that the above integrals are absolutely convergent when $\Re(s)>>0$ and in fact define rational functions of $q_E^{-s}$. We omit the details. Similar proof for other groups can be found elsewhere, say \cite{JPSS,C} in the $\GL_n\times \GL_m$ case, and \cite{GRS1, GRS2} in the symplectic group case.

\begin{rem}\upshape
 The above local zeta integrals were first studied by Gelbart and Piatetski-Shapiro in the case $\Sp_{2n}\times \GL_n$ and $\RU(n,n)\times \Res_{E/F}(\GL_n)$ in \cite{GPS}. In the symplectic group case, Ginzburg, Rallis and Soudry (\cite{GRS1, GRS2}) extended the construction to the case $\Sp_{2n}\times \GL_m$ for general $m$.  \end{rem}

\begin{prop}{\label{cor17}}
There is a meormorphic function $\gamma(s,\pi\times (\mu \eta), \psi_U), \gamma(s,\pi\times (\mu\tau),\psi_U)$ such that
$$\Psi(W, \phi_1, M(s) f_s)=\gamma(s,\pi\times(\mu \eta),\psi_U)\Psi(W,\phi_1,f_s), $$
and 
$$\Psi(W,\phi_2, M(s)\xi_s)=\gamma(s,\pi\times (\mu\tau),\psi_U)\Psi(W,\phi_2, M(s)\xi_s)$$
 for all $ W\in \CW(\pi,\psi_U),\phi_1\in \CS(E),\phi_s\in \CS(E^2)$, $ f_s\in I(s,\eta)$ and $\xi_s\in I(s,\tau)$.
\end{prop}
\begin{proof}
The local functional equation follows from the uniqueness of the Fourier-Jacobi models, \cite{GGP, Su}. See \cite{Ka} for some details of the proof in the $\Sp_{2n}$ case. 
\end{proof}


\section{Howe vectors}
In the following sections, we will follow Baruch's method given in \cite{Ba1, Ba2}, to give a proof of the local converse theorem for generic representations of $\RU(2,2)=\RU_{E/F}(2,2)$ when $E/F$ is unramified.

One main tool of Baruch's method is Howe vectors, which are used to define partial Bessel functions. In this section, we give a review of the Howe vectors in our case following \cite{Ba1, Ba2}.

From this section till the end of $\S$5, we assume the quadratic extension $E/F$ is unramified.
\subsection{Howe vectors}
Let $p_F$ be a uniformizer of $F$, which also can be viewed as a uniformizer of $E$ since $E/F$ is unramified.

Let $\psi$ (resp. $\psi_E$) be an unramified additive character of $F$ (resp. $E$).  As in $\S$1, we require that $\psi_E$ is trivial on $F$.  From these data, we have defined a character $\psi_U$ of $U$.

For a positive integer $m$, we consider the congruence subgroup $K_m=(1+\textrm{Mat}_{4\times 4}(\CP_E^m))\cap G.$ Define a character $\tau_m$ on $K_m$ by $$\tau_m(k)=\psi_E(-p_F^{-2m}k_{12})\psi(\frac{1}{2}p_F^{-2m}\tr_{E/F}(k_{24})),\quad \textrm{ for } k=(k_{ij})\in K_m. $$
It is easy to see that $\tau_m$ is indeed a character on $K_m$.

Let $$d_m=\bt(p_F^{-3m}, p_F^{-m})\in G,$$
and $H_m=d_mK_m d_m^{-1}$.
Then 
$$H_m=\begin{pmatrix}1+\CP_E^m& \CP_E^{-m}& \CP_E^{-5m}& \CP_E^{-3m}\\ \CP_E^{3m}& 1+\CP_E^m & \CP_E^{-3m}& \CP_E^{-m}\\ \CP_E^{7m}& \CP_E^{5m}& 1+\CP_E^m & \CP_E^{3m}\\ \CP_E^{5m}& \CP_E^{3m} & \CP_E^{-m} & 1+\CP_E^m \end{pmatrix}\cap G.$$

We define a character $\psi_m$ on $H_m$ by $\psi_m(j)=\tau_m(d_m^{-1}jd_m)$ for $j\in H_m$. Let $U_m=U\cap H_m$, we can check that $\psi_m|_{U_m}=\psi_U|_{U_m}$.

Now let $(\pi,V_\pi)$ be a $\psi_U$-generic irreducible smooth representation of $G=\RU(2,2)$. We fix a Whittaker functional for $\pi$ and thus for $v\in V_\pi$, there is an associated Whittaker function $W_v$. Let $v\in V_\pi$ be a vector such that $W_v(1)=1$. For $m\ge 1$, as Barach defined in \cite{Ba1, Ba2}, we consider 
$$v_m=\frac{1}{\Vol(U_m)}\int_{U_m}\psi_U(u)^{-1}\pi(u)v du.$$
Let $C=C(v)$ be an integer such that $v$ is fixed by $K_C$.
\begin{lem}{\label{lemma21}}
 We have
\begin{enumerate}
\item $W_{v_m}(1)=1.$
\item If $m\ge C$, we have $\pi(j)v_m=\psi_m(j)v_m$ for all $j\in H_m$.
\item If $k\le m$, then 
$$v_m=\frac{1}{\Vol(U_m)}\int_{U_m}\psi_U(u)^{-1}\pi(u)v_k du.$$
\end{enumerate}
\end{lem}
\begin{proof}
The proof is given in \cite{Ba1} in the general case.  Although \cite{Ba1} is not published, the proof in the general case is in fact the same as the $\RU(2,1)$ case, which can be found in \cite{Ba2}.
\end{proof}
From Lemma \ref{lemma21} (2), we have
\begin{equation}\label{eq21}
W_{v_m}(ugj)=\psi_U(u)\psi_m(j)W_{v_m}(j), \forall u\in U, j\in H_m, m\ge C.
\end{equation}
The vectors $\wpair{v_m}_{m\ge C}$ are called Howe vectors and $W_{v_m}$ are the associated partial Bessel functions.

\begin{lem}{\label{lemma22}}
\begin{enumerate}
\item For $m\ge C$, $t\in T$, if $W_{v_m}(t)\ne 0, $ then $\alpha(t)\in 1+\CP_E^m$ and $\beta(t)\in 1+\CP_F^m$, where $\alpha(t)=a_1/a_2$ and $\beta(t)=a_2\bar a_2$ for $t=\bt(a_1,a_2)\in T$.
\item For $w=s_\alpha, s_\beta, s_\alpha s_\beta$ or $s_\beta s_\alpha$, we have $W_{v_m}(t\dot w)=0$ for all $t\in T$ and $m\ge C$.
\end{enumerate}
\end{lem}
\begin{proof}
(1) For $x\in \CP_E^{-m}$, consider the element $\bx_\alpha(x)\in U_m$. 
We have $$t\bx_\alpha(x)=\bx_{\alpha}(\alpha(t)x)t.$$
Thus by Lemma 2.1 or Eq.(\ref{eq21}), we have $$\psi_m(\bx_{\alpha}(x))W_{v_m}(t)=\psi_U(\bx_{\alpha}(\alpha(t) x))W_{v_m}(t).$$
If $W_{v_m}(t)\ne 0$, we get $\psi_m(\bx_\alpha(x))=\psi_U(\bx_\alpha(\alpha(t)x))$, or $\psi_E(x)=\psi_E(\alpha(t)x)$. Since this is true for all $x\in \CP_E^{-m}$ and $\psi_E$ is unramified, we get $\alpha(t)-1\in \CP_E^m$, or $\alpha(t)\in 1+\CP_E^m$. A similar argument shows that $\beta(t)\in 1+\CP_F^m$. This proves (1).\\
(2) Note that $w$ which is given in the condition sends a simple root $\gamma$ to a positive non-simple root $w(\gamma)$. In fact, we have $s_\alpha(\beta)=2\alpha+\beta, s_\beta(\alpha)=\alpha+\beta$, $s_\alpha s_\beta(\alpha)=\alpha+\beta$ and $s_\beta s_\alpha (\beta)=2\alpha+\beta$. Take $r$ such that $\bx_\gamma(r)\in U_m$, by Eq.(\ref{eq21}), we have
$$\psi_U(r)W_{v_m}(t\dot w)=W_{v_m}(t\dot w \bx_{\gamma}(r))=W_{v_m}(\bx_{w(\gamma)}(r \gamma (t) )t\dot w)=W_{v_m}(t\dot w).$$
We can take $r$ such that $\psi_U(r)\ne 1$. Thus $W_{v_m}(tw)=0$. 
\end{proof}
\begin{cor}{\label{cor23}}
For $m\ge C$, and $t=\bt(a_1,a_2)\in T$, if $W_{v_m}(t)\ne 0$, then $$a_1/a_2\in 1+\CP_E^m, \textrm{ and } a_2\in E^1(1+\CP_E^m).$$
\end{cor}
\begin{proof}
By Lemma {\ref{lemma22}}, we have $a_1/a_2=\alpha(t)\in 1+\CP_E^m$ and $a_2\bar a_2\in 1+\CP_F^m$. Since $E/F$ is unramified, we have $\Nm_{E/F}(1+\CP_E^m)=1+\CP_F^m$, see Chapter V, $\S$2 of \cite{Se}. Since $\Nm(a_2)=a_2\bar a_2\in 1+\CP_F^m$, there exists a $b\in 1+\CP_E^m$ such that $b\bar b =a_2\bar a_2$. It is clear that $a_2/b\in E^1$, and thus $a_2\in E^1(1+\CP_E^m)$.
\end{proof}

We fix the following notations:\\
\noindent\textbf{Notations:} Let $(\pi,V_\pi)$ and $(\pi',V_{\pi'})$ be two irreducible admissible $\psi_U$-generic representations of $G$ such that $\omega_\pi=\omega_{\pi'}$, where $\omega_\pi$ is the central character of $\pi$. Let $v\in V_\pi$, $v'\in V_{\pi'}$, such that $W_v(1)=W_{v'}(1)=1$. We have defined Howe vectors $v_m, v'_m$ for $m\ge C$, where $C=C(v,v')$ is an integer such that $v$ is fixed by $\pi(K_C)$ and $v'$ is fixed by $\pi'(K_C)$.
\begin{cor}{\label{cor24}}
If $m\ge L$, we have $W_{v_m}(b)=W_{v_m'}(b)$ for all $b\in B$.
\end{cor}
\begin{proof}
For $b\in B$, we can write $b=ut$ with $u\in U,t\in T$. Since $W_{v_m}(ut)=\psi_U(u)W_{v_m}(t)$ and $W_{v_m'}(ut)=\psi_U(u)W_{v'_m}(t)$, it suffices to prove that $W_{v_m}(t)=W_{v'_m}(t)$ for all $t\in T$. Suppose $t=\bt(a_1,a_2)$. By Corollary \ref{cor23}, if $a_2\notin E^1(1+P_E^m),$ or $a_1a_2^{-1}\notin 1+\CP_E^m$, we have $W_{v_m}(t)=0=W_{v_m'}(t)$. Now suppose that $a_2\in E^1(1+\CP_E^m)$ and $a_1a_2^{-1}\in 1+\CP_E^m$. We can write $a_2=za_2'$ with $z\in E^1,a_2'\in 1+P_E^m$. Then 
$$t=\bt(z,z)\bt(a_2'a_1a_2^{-1},a_2').$$
Since $\bt(z,z)$ is in the center $Z$ of $G$, we have $$W_{v_m}(t)=\omega_\pi(z)W_{v_m}(\bt(a_2'a_1a_2^{-1}, a_2')).$$
Since $a_2'\in 1+\CP_E^m, a_2'a_1a_2^{-1}\in 1+\CP_E^m$, we have $\bt(a_2'a_1a_2^{-1}, a_2')\in H_m$. Thus by Lemma \ref{lemma21} (2), $W_{v_m}(t')=\psi_m(\bt(a_2'a_1a_2^{-1}, a_2'))W_{v_m}(1)=W_{v_m}(1)=1$. Then $W_{v_m}(t)=\omega_\pi(z)$. The same argument shows that $W_{v_m'}(t)=\omega_{\pi'}(z)$. Since $\omega_\pi=\omega_{\pi'}$, we get $W_{v_m}(t)=W_{v_m'}(t)$.
\end{proof}

\subsection{A stability property of the Whittaker functions associated with Howe vectors.} 

We recall the Bruhat order on the Weyl group $\bW$, see \cite{Hu} for example. An element $w\in \bW$ can be written as a product of simple roots in a minimal length, say $w=s_{\gamma_1}\dots s_{\gamma_l}$ with $l=\textrm{length}(w)$. We say $w'\le w$ if $w'$ can be written as a product of sub-expression, i.e., 
$$w'=s_{\gamma_{i_1}}\dots s_{\gamma_{i_k}},$$
with $1\le i_1<i_2<\dots <i_k\le l$. This definition is independent of the choice of the minimal expression of $w$. We say $w'<w$ if $w'\le w$ and $w'\ne w$.

For $w\in \bW$, denote $U_w^+=\wpair{u\in U: wuw^{-1}\in U}$ and $U_w^-=\wpair{u\in U: wuw^{-1}\notin U}.$ Then
$$ U_w^+=\prod_{\gamma\in \Sigma_w^+}U_\gamma, \textrm{ and } U_w^-=\prod_{\gamma\in \Sigma_w^-}U_\gamma,$$
where $\Sigma_w^+=\wpair{\gamma\in \Sigma^+, w(\gamma)>0}$ and $\Sigma_w^-=\wpair{\gamma\in \Sigma^+, w(\gamma)<0}$.  Recall that $U_{w,m}^-=U_w^-\cap H_m$.

\begin{prop}\label{prop25}
Given $w\in \bW$. Let $a_t, 0\le t \le l(w)$ be a sequence of integers with $a_0=0$ and $a_t\ge t+a_{t-1}$ for all $t$ with $1\le t \le l(w)$. Let $m$ be an integer such that $m\ge 4^{a_{l(w)}}C$.
\begin{enumerate}
\item If $W_{v_k}(t\dot w')=W_{v_k'}(t \dot w'),$ for all $ w'<w , k\ge 4^{a_{l(w')}} C$, and $t\in T$, then
$$W_{v_m}(t\dot wu_w^-)=W_{v_m'}(t \dot wu_w^-),$$
for all $u_w^-\in U_w^--U_{w,m}^-$.
\item If $W_{v_k}(t\dot w')=W_{v_k'}(t\dot w'),$ for all $ w'\le w , k\ge 4^{a_{l(w')}} C$, and $t\in T$, then
$$W_{v_m}(g)=W_{v_m'}(g),$$
for all $g\in BwB$.
\end{enumerate}
\end{prop}
\noindent \textbf{Remark:} We can take the sequence $a_t=t^2$ as in \cite{Ba1}. We can also take the sequence $a_0=0, a_1=1, a_2=3, a_3=6$ and $a_4=10$ in our case.

\begin{proof}This is essentially Lemma 6.2.6 of \cite{Ba1}. Since \cite{Ba1} is not published, we explain more about the proof. For $w\le w_2=s_\alpha s_\beta s_\alpha$, Proposition \ref{prop25} is proved in a more general setting in Theorem 3.11 \cite{Zh2}, which in fact justified an ambiguity in the original proof in Lemma 6.2.6 of \cite{Ba1}. Now we consider the case when $w=w_1=s_\beta s_\alpha s_\beta,$ and $w=w_0=(s_\alpha s_\beta)^2$.  We order the set $\Sigma^-$ by height, i.e., we denote $\gamma_1=\alpha, \gamma_2=\beta, \gamma_3=\alpha+\beta, \gamma_4=2\alpha+\beta$. Here the order of $\alpha$ and $\beta$ are not important.

From the proof of Theorem 3.11 \cite{Zh2}, to prove our proposition for $w=w_0$ or $w_1$, it suffices to check the following Claim $(*)$:

Claim $(*)$: Suppose that $w=w_1$ or $w_0$. Given $g=tw \bx_{\gamma_t}(r_k)\dots  \bx_{\gamma_i}(r_i)\in G$ with $t\in T, r_i\ne 0, $ and $w(\gamma_i)<0$, where the subscript of $\gamma$ is decreasing, i.e., $k> k-1\ge \dots >i$, then $g\bx_{-\gamma_i}(-r_i^{-1})\in Bw'B$ for some $w'<w$.

To prove this Claim, we need to use the Chevalley relation $ \bx_{\gamma_i}(r_i) \bx_{-\gamma_i}(-1/r_i)\in s_{\gamma_i}B$, where $s_{\gamma_i}$ is the reflection associated with $\gamma_i$. For this relation, see \cite{St}.

We will check Claim $(*)$ case by case. First suppose that $i=1 $ or $2$, so that $\alpha_i$ is simple. Then by the above Chevalley relation, we have
$$g\bx_{-\gamma_i}(-r_i^{-1})=tw \bx_{\gamma_t}(r_t) \dots \bx_{\gamma_{i+1}}(r_{i+1})s_{\gamma_i}b, \textrm{ for some } b\in B. $$
Using the relation $s_{\gamma_i}^{-1}\bx_{\gamma_j}(r_j)s_{\gamma_i}=\bx_{s_{\gamma_i}(\gamma_j)}(r_j)$, we get 
$$ g\bx_{-\gamma_i}(-r_i^{-1})=tws_{\gamma_j} \bx_{s_{\gamma_i}(\gamma_t)}(r_t) \dots \bx_{s_{\gamma_i}(\gamma_{i+1})}(r_{i+1})b.$$
Since $s_{\gamma_i}(\gamma_j)>0$, we get $ \bx_{s_{\gamma_i}(\gamma_j)}(r_j)\in U$, and thus 
$$g\bx_{-\gamma_i}(-r_i^{-1})\in Bws_{\gamma_i} B. $$
Now the assertion follows since $w'=ws_{\gamma_i}<w$ by the assumption $w (\gamma_i)<0$.

Next we consider the case $i=3$, so that $g=tw\bx_{2\alpha+\beta}(r_4)\bx_{\alpha+\beta}(r_3)$. We can check that $s_{\alpha+\beta}=s_\beta s_\alpha s_\beta=w_1$. Using the above Chevalley relation, we can get 
$$g\bx_{-(\alpha+\beta)}(-1/r_3)\in B w U_{2\alpha+\beta} w_1 B.$$
Note that $w_1=w_1^{-1}$ and $w_1U_{2\alpha+\beta}w_1=U_{w_1(2\alpha+\beta)}=U_{-\beta}\subset Bs_\beta B$, thus we get 
$$ g\bx_{-(\alpha+\beta)}(-1/r_3)\in Bww_1 Bs_\beta B.  $$
If $w=w_1$, then $ g\bx_{-(\alpha+\beta)}(-1/r_3)\in Bw' B$ with $w'=s_\beta<w_1$. If $w=w_0$, then $ww_1=s_\alpha$, and thus $g\bx_{-(\alpha+\beta)}(-1/r_3)\in Bs_\alpha Bs_\beta B=Bs_\alpha s_\beta B$. The assertion follows with $w'=s_\alpha s_\beta<w_0$.

Finally we consider the case $i=4$, so that $g=tw \bx_{2\alpha+\beta}(r_4)$. We have $s_{2\alpha+\beta}=s_\alpha s_\beta s_\alpha=w_2$. Thus from the Chevalley relation, we get
$$ g\bx_{-(2\alpha+\beta)}(-r_4^{-1})\in Bww_2 B.$$
It suffices to check that $w'=ww_2<w$. In fact, we have $w_1w_2=s_\alpha s_\beta<w_1$, and $w_0w_2=s_\beta<w_0$. The proof of Claim $(*)$ and hence the proposition is complete.
\end{proof}

As a direct consequence of Proposition \ref{prop25}, Lemma \ref{lemma22} (2) and Corollary \ref{cor24}, we have the following
\begin{cor}\label{cor26} 
 Let $a_t$ be a sequence as in Proposition $\ref{prop25}$. Then
\begin{enumerate}
\item if $w=1, s_\alpha, s_\beta, s_\alpha s_\beta, s_\beta s_\alpha$, we have $$W_{v_m'}(g)=W_{v_m'}(g),$$
for all $ g\in BwB$  and $m\ge 4^{a_{l(w)}}C$; 
\item if $w=w_1, w_2$, we have
$$W_{v_m}(t \dot w u^-_w)=W_{v_m'}(t \dot w u^-_w),$$
for $u_w^-\in U_w^- -U_{w,m}^-$ for $m\ge 4^{a_{l(w)}}C$ and all $t\in T$.
\end{enumerate}
\end{cor}

\section{Induced representations and intertwining operator}
In this section, we will construct some sections in the induced representations $I(s,\eta)$ and $I(s,\tau)$ for a given quasi-character $\eta$ of $E^\times$ and an irreducible smooth representation $\tau$ of $\GL_2(E)$. Since the construction for sections in $I(s,\eta)$ is quite similar in the case $I(s,\tau)$ and the proof is easier, we only write down the statement and the proof in the case $I(s,\tau)$.

Let $\hat N$ be the opposite of $N$, i.e., $\hat N$ consists of matrices of the form 
$$\hat \bn(b):=\begin{pmatrix}1& \\ b&1 \end{pmatrix}, b\in \Herm_2(F).$$
Recall that $N^{(2)}$ is the unipotent subgroup of $\GL_2(E)$, and under the embedding $\GL_2(E)\cong M\incl G$, we have $N^{(2)}\cong U_{\alpha}$.

Let $X$ be an open compact subgroup of $N$. For $x\in X,i>0$, define $A(x,i)=\wpair{\hat n\in \hat N: \hat n x\in P\cdot \hat N_i}$.
\begin{lem}{\label{lemma31}}
\begin{enumerate}
\item For any positive integer $c$, there exists an integer $i_1=i_1(X,c)$ such that for all $i\ge i_1$,  $x\in X$ and $\hat n\in A(x,i)$, we can write $$\hat n x=n\bm(a)\hat n_0,$$
with $n\in N, \hat n_0\in \hat N_i$ and $a\in K_c^{(2)}:=1+\Mat_{2\times 2}(\CP_E^c)$.

\item There exists an integer $i_0=i_0(X)$ such that $A(x,i)$ is independent of $x$ for all $i\ge i_0$, i.e., $A(x_1,i)=A(x_2,i)$ for all $x_1,x_2\in X$. In fact, we can choose $i_0=i_0(X)$ such that $A(x,i)=\hat N_i$ for all $i\ge i_0$. Here $\hat N_i=H_i \cap \hat N.$
\end{enumerate}
\end{lem}
In the $\GL_n$ case, this is Lemma 4.1 of \cite{Ba1}. The proof in our case is similar. 
\begin{proof}
 Since $X$ is compact, there is a constant $D_X$ such that $|x_{l,j}|<D_X$, for all $x=\bn((x_{l,j}))\in X\subset N$. 

For $x\in X, \hat n\in A(x,i)$, we assume that $\hat n x=p\hat y^{-1}$ with $p\in P, \hat y^{-1}\in \hat N_i$. We have 
$$\hat n^{-1}p= x \hat y$$
Let $$\hat n^{-1}=\begin{pmatrix}1& \\ \hat b& 1 \end{pmatrix}=\hat \bn(\hat b),p=\bn(b')\bm(a), \textrm{ for }\hat b, b'\in \Herm_2(F), a\in \GL_2(E).$$
Then $$\hat n^{-1}p=\hat \bn(\hat b)\bn(b')\bm(a)=\begin{pmatrix}a& b' {}^t\bar a^{-1}\\ \hat ba& (\hat b b'+1){}^t\bar a^{-1}\end{pmatrix}$$
On the other hand, if we assume $x=\bn(x_0), \hat y=\hat \bn(\hat y_0)$ with $x_0, \hat y_0\in \Herm_2(F)$, then we have
$$x\hat y=\begin{pmatrix} 1& x_0\\ &1\end{pmatrix}\begin{pmatrix}1& \\ \hat y_0 & 1 \end{pmatrix}=\begin{pmatrix}1+x_0\hat y_0 & x_0\\ \hat y_0 &1\end{pmatrix}.$$ 
From $\hat n^{-1}p= x\hat y$, we get $$a=1+x_0\hat y_0$$ and 
$$ \hat b=\hat y a^{-1}=\hat y_0 (1+x_0\hat y_0)^{-1}.$$
Since the entries of $x_{0}$ are bounded, and entries of $\hat y_{0}$ go to zero as $i\ra \infty$. Thus for any positive integer $c$, we can take $i_1=i_1(X,c)$ such that if $i\ge i_1$, we have $$a=1+x_0 \hat y_0\in K^{(2)}_c.$$ Thus (1) follows.

To prove (2), we write $\hat b=(\hat b_{jk})$, $x_0=(x_{jk})$ and $\hat y_0=(\hat y_{jk})$, $j,k=1,2$. By Cramer's rule, we have
$$(1+x_0\hat y_0)^{-1}=\det(1+x_0 \hat y_0)^{-1}\begin{pmatrix} 1+x_{21}\hat y_{12}+x_{22}\hat y_{22} & -x_{11}\hat y_{12}-x_{12}\hat y_{22}\\ -x_{21}\hat y_{11}-x_{22}\hat y_{21} & 1+x_{11}\hat y_{11}+x_{12}\hat y_{21}\end{pmatrix}.$$
From $\hat b=\hat y_0 (1+x_0\hat y_0)^{-1}$, we can solve that
\begin{equation}\label{eq31}\hat b=\det(1+x_0\hat y_0)^{-1}\begin{pmatrix}\hat y_{11}+\det(\hat y_0) x_{22} & \hat y_{12}-\det(\hat y_0) x_{12} \\
\hat y_{21}-\det(\hat y_0) x_{21}& \hat y_{22}+\det(\hat y_0) x_{11}  \end{pmatrix} \end{equation}

Note that $\hat y\in \hat N_i$, i.e., 
\begin{equation}\label{eq32}\hat y_0\in \begin{pmatrix}\CP_E^{7i} & \CP_E^{5i}\\ \CP_E^{5i} & \CP_E^{3i} \end{pmatrix}.\end{equation}
We have $\det(\hat y_0)\in \CP_E^{10 i}$. We can choose $i_2=i_2(X)$ large enough such that for $i\ge i_2$, we have
$$q_E^{-10i}D_X\le q_E^{-7i}.$$
Then \begin{equation}\label{eq33}\det(\hat y_0)x_{jk}\in \CP_E^{7i}, i\ge i_2, \end{equation}  for all $j,k=1,2$. Now we take $i_0(X)=\max\wpair{i_2(X), i_1(X,1)}$. Then for $i\ge i_0(X)$, we have $\det(1+x_0\hat y_0)\in 1+\CP_E\subset \CO_E^\times$. From Eq. (\ref{eq31}) and Eq. (\ref{eq33}), we get
$$ \hat b\in  \begin{pmatrix}\CP_E^{7i}+\CP_E^{7i} & \CP_E^{5i}+\CP_E^{7i}\\ \CP_E^{5i} +\CP_E^{7i}& \CP_E^{3i}+\CP_E^{7i} \end{pmatrix} \subset  \begin{pmatrix}\CP_E^{7i} & \CP_E^{5i}\\ \CP_E^{5i} & \CP_E^{3i} \end{pmatrix}.$$
Thus $\hat n^{-1}=\hat \bn(\hat b)\in \hat N_i,$ and hence $\hat n\in \hat N_i$. This shows $A(x,i)\subset \hat N_i$. 

  To show every element $\hat y\in \hat N_i$ is contained in $A(x,i)$ for $i$ large enough. As above, we write $\hat y=\hat \bn(\hat y_0)$ and $x=\bn(x_0)$. We first notice that, for $ i$ large enough, we can assume $\det(1+\hat y_0 x_0 )\ne 0$. From this, it is easy to check that $\hat y x\in P\hat N$, and thus we can write $\hat y x=p \hat n$ for $p\in P, \hat n\in \hat N$. A similar argument as above will show that $\hat n\in \hat N_i$ for $i$ large. In fact, we only have to switch the role of $\hat y$ and $\hat n$ in the above argument. Thus $\hat y\in A(x,i)$ for $i$ large. This finishes the proof of (2).
\end{proof}

 For $i>0$ and a vector $\epsilon\in V_\tau$, we consider a function $\xi_s^{i,\epsilon}:G\ra V_\tau$ defined by 
$$\xi_s^{i,\epsilon}(g)=\left\{\begin{array}{lll}|\det(a)|^{s+1/2}_E \tau(a)\epsilon,& \textrm{ if } g=n\bm(a)\hat n, n\in N, a\in \GL_2(E), \hat n\in \hat N_i,\\ 0, &\textrm{ otherwise.}\end{array}\right.$$

\begin{lem}{\label{lemma32}}
There is an integer $i_2=i_2(\epsilon)$ such that if $ i\ge i_2$, $\xi_s^{i,\epsilon}$ defines an element in $I(s,\tau)=\Ind_P^G(\tau_{s-1/2})$, where $\tau_{s-1/2}=\tau |\det|^{s-1/2}$.
\end{lem}
\begin{proof}
It is clear that $\xi_s^{i,\epsilon}(n\bm(a)g)=\delta_P^{1/2}(a)\tau_{s-1/2}(a)\xi_s^{i,\epsilon}(g)$, where $\delta_P$ is the modulus character of the Siegel parabolic subgroup $P$. We need to check that for $i$ large enough, there is an open compact subgroup $H_{i,\epsilon}\subset G$ such that $$\xi_s^{i,\epsilon}(gh)=\xi_s^{i,\epsilon}(g), \forall h\in H_{i,\epsilon}.$$
Let $c>0$ be an integer such that $\epsilon$ is fixed by $K_c^{(2)}=1+\textrm{Mat}_{2\times 2}(\CP_E^c)\subset \GL_2(E)$ under $\tau$. Recall that $K_c$ is used to denote the congruence subgroup $G\cap (1+\Mat_{4\times 4}(\CP_E^c))$ of $G$. Let $i_2(\epsilon,\tau)=\max\{c,i_0(N\cap K_c),i_1(N\cap K_c,c)\}$, where $i_0$ and $i_1$ are defined in Lemma \ref{lemma31}, and $N_c=N\cap J_c$ is defined in $\S$3. By the Lemma {\ref{lemma31}}, we have $A(x,i)=\hat N_i$ for $x\in N_c$, $i\ge i_2(\epsilon,\tau)$.

We take $H_{i,\epsilon}=K_{7i}$. We have the decomposition $K_{7i}=(\hat N\cap K_{7i}) (M\cap K_{7i})(N\cap K_{7i})$. For $h\in \hat N\cap K_{7i}\subset \hat N_i$, it is clear that $\xi^{i,\epsilon}_s(gh)=\xi^{i,\epsilon}_s(g)$. For $h=\bm(a_0)\in M\cap K_{7i}$ with $a_0\in K^{(2)}_{7i}= 1+\textrm{Mat}_{2\times 2}(\CP_E^{7i})$, and $\hat n \in \hat N_i$, we have
$$\bm(a_0)^{-1}\hat n \bm(a_0)\in \hat N_i$$
by an explicit calculation.

From this we can check that $g\in P\hat N_i$ if and only if $g\bm(a_0)\in P\hat N_i$. Moreover,
$$\xi_s^{i,\epsilon}(n\bm(a)\hat n \bm(a_0))=\xi_s^{i,\epsilon}(n\bm(aa_0) \bm(a_0^{-1})\hat n \bm(a_0))$$
$$=|\det(aa_0)|^{s+1/2}\tau(aa_0)\epsilon= |\det(a)|^{s+1/2}\tau(a)\epsilon=\xi_s^{i,\epsilon}(n\bm(a) \hat n),$$
where we used $i>c$ and hence $\tau(a_0)\epsilon=\epsilon$ for $a_0\in K^{(2)}_{7i}$.

Next, we assume that $h\in N\cap K_{7i}\subset N\cap K_c$. By the above lemma and our assumption, we have $$A(h,i)=A(h^{-1},i) =\hat N_i.$$
In particular, for $\hat n\in \hat N_i$, we have $\hat n h\in P\cdot \hat N_i$ and $\hat n h^{-1}\in P\cdot \hat N_i$. Thus $g\in P\hat N_i$ if and only if $gh\in P\hat N_i$. Moreover, by Lemma \ref{lemma31} (1), we have $\hat n h=n_0\bm(a_0)\hat n_0$ with $a_0\in K_c^{(2)}$, then we have
$$\xi_s^{i,\epsilon}(n\bm(a)\hat n h)=\xi_s^{i,\epsilon}(n\bm(a)n_0\bm(a^{-1})\bm(aa_0)\hat n_0)$$
$$=|\det(aa_0)|^{s+1/2}\tau(aa_0)\epsilon=|\det(a)|^{s+1/2}\tau(a)\epsilon=\xi_s^{i,\epsilon}(n\bm(a)\hat n).$$
This shows that $\xi_s^{i,\epsilon}(gh)=\xi_s^{i,\epsilon}(g)$ for all $h\in K_{7i}$. The proof is complete.
\end{proof}
Consider the function $f_s^{i,\epsilon}(g):=f_{\xi^{i,\epsilon}_s}(g)=\lambda(\xi^{i,\epsilon}_s(g)),$ where $\lambda$ is a fixed Whittaker functional of $\tau$.  Then $\supp f_s^{i,\epsilon}=P\cdot \hat N_i$ and we have
\begin{equation}\label{eq34}f^{i,\epsilon}_s(n\bm(a) \hat n)=|\det(a)|^{s+1/2}W^{(2)}_{\epsilon}(a),\end{equation}
where $W^{(2)}_\epsilon(a)=\lambda(\tau(a)\epsilon)$ is the Whittaker function of $\tau$ associated to the vector $\epsilon\in V_\tau$.

Now we consider $\tilde \xi^{i,\epsilon}_{1-s}=M(s)\xi^{i,\epsilon}_s\in I(1-s,\tau^*)$.  By definition, we have
$$\tilde \xi^{i,\epsilon}_{1-s}(g)=\int_{N}\xi_s^{i,\epsilon}(\dot w_1^{-1}n g)dn.$$
Let $X$ be an open compact subgroup of $N$, we evaluate $\tilde \xi$ at $\dot w_1 x$ for $x\in X$. 
\begin{prop}{\label{prop33}}
There is an integer $I=I(X,\epsilon)$ such that for $i\ge I$, we have $\tilde \xi^{i,\epsilon}_{1-s}(w_1x)=\vol(\hat N_i)\epsilon$ for all $x\in X$.
\end{prop}
\begin{proof}
We have $$\tilde \xi^{i,\epsilon}_{1-s}(\dot w_1x)=\int_{N}\xi_s^{i,\epsilon}(\dot w_1^{-1}n \dot w_1x)dn.$$
Again, let $c$ be a positive integer such that $\epsilon$ is fixed by $\tau(K_c^{(2)})$. Let $I=\max\wpair{i_0(X),i_1(X,c)}$.
By definition of $\xi^{i,\epsilon}_s$ and Lemma \ref{lemma31}, for $i\ge I$, we have $\xi_s^{i,\epsilon}(\dot w_1^{-1}n\dot w_1x)\ne 0$ if and only if $\dot w_1^{-1}n\dot w_1x\in P\cdot \hat N_i$ if and only if $\dot w_1^{-1}n\dot w_1\in \hat N_i$. By part (1) of Lemma \ref{lemma31}, if $\dot w_1^{-1}n\dot w_1\in A(x,i)=\hat N_i$, we have $$\dot w_1^{-1}n\dot w_1x=n'\bm(a)\hat n, \textrm{ with } n'\in N, a\in K_{c}^{(2)}, \hat n\in \hat N_i.$$
Since for $a\in K_c^{(2)}$, we have $\tau(a)\epsilon=\epsilon$ and $|\det(a)|=1$, we have $$\xi_s^{i,\epsilon}(\dot w_1^{-1}n\dot w_1x)=\left\{\begin{array}{lll} |\det(a)|^{s+1/2} \tau(a)\epsilon=\epsilon, & \textrm{ if } \dot w_1^{-1}n \dot w_1\in \hat N_i \\ 0, & \textrm{ otherwise.}\end{array} \right.$$ 
Thus $$\tilde \xi_{1-s}^{i,\epsilon}(\dot w_1x)=\vol(\hat N_i) \epsilon.$$
This proves the assertion.
\end{proof}
Let $\tilde f^{i,\epsilon}_{1-s}(g)=f_{\tilde \xi_{1-s}^{i,\epsilon}}(g)=\lambda( \tilde \xi_{1-s}^{i,\epsilon}(g))$.  Then
 \begin{equation}{\label{eq35}}\tilde f^{i,\epsilon}_{1-s}(n\bm(a)\dot w_1x)=\vol(\hat N_i)|\det(a)|^{3/2-s}\tilde W_{\epsilon}^{(2)}(a), \quad i>I(X,\epsilon), x\in X.\end{equation} 
 Here $\tilde W^{(2)}$ denotes the Whittaker function for the representation $\tau^*=w_1\tau$ of $M\cong \GL_2(E)$, see $\S$1.

\section{Twisting by characters of $E^\times$}
We keep the notations of $\S$2. In particular, $E/F$ is unramified, $\psi$ is an unramified additive character of $F$, $\pi,\pi'$ are two $\psi_U$-generic irreducible smooth representations of $\Sp_4(F)$ with the same central character. We also fixed $v\in V_\pi, v'\in V_{\pi'}$ and a positive integer $C=C(v,v')$. We have defined Howe vectors $v_m, v_m'$. Let $\mu$ be a fixed character of $E^\times$ such that $\mu|_{F^\times}=\epsilon_{E/F}$.
\begin{prop}\label{prop41}
 If $\gamma(s,\pi\times \mu\eta, \psi)=\gamma(s,\pi'\times \mu\eta, \psi)$ for all quasi-characters $\eta$ of $F^\times$, then 
$$W_{v_m}(t\dot w_2)=W_{v_m'}(t\dot w_2),$$
for all $t\in T, m\ge 4^6C$. Recall that $w_2=s_\alpha s_\beta s_\alpha.$
\end{prop}

\begin{proof}
An almost identical argument as the proof of Theorem 4.4 of \cite{Zh2} will give us the following
\begin{align}
&\gamma(s,\pi\times\mu \eta, \psi)-\gamma(s,\pi'\times\mu \eta,\psi) \label{eq41}\\
=&\epsilon_{\psi^{-1}}q_F^{5k}q_E^{4k}\int_{E^\times}(W_{v_k}(\bt(a)\dot w_2)-W_{v_k'}(\bt(a)\dot w_2))\mu(a)\eta_{-s-1/2}^{-1}(a)da, \nonumber
\end{align}
for $k=4^6C$. Recall that $\epsilon_{\psi^{-1}}$ is the Weil index appeared in the Weil representation formula Eq.(\ref{eq11}),  $\bt(a)=\diag(a,1,1,a^{-1})$, $\eta_{-s-1}(a)=\eta(a)|a|^{-s-1}$. In fact, in the $\Sp_{2n}$ case, Eq.(\ref{eq41}) (with a little bit modification on the constant term $q_F^{5k}q_E^{4k}$ and the exponent $-s-1/2$) is Eq.(4.4) of \cite{Zh2} with $k=4^{l(w_2)^2}C=4^9C$, where we used the sequence $a_t=t^2$ in Proposition \ref{prop25}. It is easy to see that this is true for $k=4^6C$ by choosing the sequence $a_0=0,a_1=1, a_2=3, a_3=6$ and $a_4=10$. 

By the assumption on $\gamma$-factors and Eq(\ref{eq41}), we have
$$ \int_{E^\times}(W_{v_k}(\bt(a)\dot w_2)-W_{v_k'}(\bt(a)\dot w_2))\mu(a)\eta_{-s-1/2}^{-1}(a)da\equiv 0,$$
for any quasi-character $\eta$ of $F^\times$. By the inverse Mellin transformation, we get
$$W_{v_k}(\bt(a)\dot w_2)=W_{v_k'}(\bt(a)\dot w_2), \forall a\in E^\times.$$
From this, we can prove $W_{v_k}(t\dot w_2)=W_{v'_k}(t\dot w_2)$ for all $t=\bt(a,b)$ as follows. Take $r\in \CP_F^{-k}$, so that $\bx_\beta(r)\in U_k$. From Eq.(\ref{eq21}) and the relation
$$t \dot w_2 \bx_\beta(r)= t\bx_\beta(r) \dot w_2= \bx_\beta(b^2 r) t \dot w_2,$$
we get $\psi(r)W_{v_k}(t \dot w_2)=\psi(b^2r)W_{v_k}(t\dot w_2)$. Thus, if $W_{v_k}(t\dot w_2)\ne 0$, we have $\psi(r)=\psi(b^2r)$ for all $r\in \CP_F^{-k}$, which implies $b^2\in 1+\CP_F^k$. Thus we get $b=e b_1$ with $b_1\in 1+\CP_E^k$ and $e\in E^1$ (see the proof of Corollary \ref{cor23}). Thus we can write $t=e \bt(a_1)t_1$, with $a_1=ea$ and $t_1=\bt(1, b_1)$. Since
$$t\dot w_2 = e \bt(a_1) t_1 \dot w_2 =e \bt(a_1)\dot w_2 t_1,$$
and $t_1\in H_k$, we get 
$$W_{v_k}(t\dot w_2)=\omega_\pi(e) W_{v_k}(\bt(a_1)\dot w_2),$$
by Eq.(\ref{eq21}). Here $\omega_\pi$ is the central character of $\pi$. Since $ \pi$ and $\pi'$ have the same central character, we get $W_{v_k}(t\dot w_2)=W_{v'_k}(t \dot w_2)$ for all $t\in T$. By Lemma \ref{lemma21} (3) and Proposition \ref{prop25}, it is easy to check that
$$W_{v_m}(t\dot w_2)=W_{v'_m}(t\dot w_2), \forall m\ge k=4^6C, t\in T.$$
This concludes the proof.
\end{proof}
\begin{prop}\label{prop42}
 Suppose that $\gamma(s,\pi\times \mu\eta, \psi)=\gamma(s,\pi'\times \mu\eta, \psi)$ for all quasi-characters $\eta$ of $F^\times$. Then for $m\ge 4^9C$, $a\in \GL_2(E)$ and $n\in N-N_m$, we have 
$$W_{v_m}(\bm(a)\dot w_1 n)=W_{v_m'}(\bm(a)\dot w_1 n).$$
\end{prop}
Recall that $N_m=N\cap H_m$.
\begin{proof}
The proof is similar to the proof of Theorem 3.11 of \cite{Zh2} and the method is due to Baruch, \cite{Ba1}. We give the details here. 

By Corollary \ref{cor26}, we have $W_{v_k}(g)=W_{v_k'}(g)$ for all $g\in BwB$ with $w\in \bW$ with $\textrm{length}(w)\le 2$ and $k\ge 4^{l(w)^2}C$. From the condition and Proposition \ref{prop41}, we get $W_{v_k}(t\dot w_2)=W_{v_k'}(t\dot w_2)$ for $t\in T, k\ge 4^6C$. Thus we have $W_{v_k}(g)=W_{v_k'}(g)$ for $g\in Bw_2B$ and $k\ge 4^6C$.

For $n\in N$, we can write $n=\bx_{2\alpha+\beta}(r_3)\bx_{\alpha+\beta}(r_2)\bx_{\beta}(r_1)$ for $r_1, r_3\in F, r_2\in E$. Denote $\gamma_3=2\alpha+\beta, \gamma_2=\alpha+\beta, \gamma_1=\beta$. Let $j$ ($1\le j\le 3$) be the first index such that $\bx_{\gamma_j}(r_j)\notin N_m$. Then it suffices to show that
$$W_{v_m}(\bm(a)\dot w_1 \bx_{\gamma_3}(r_3)\dots \bx_{\gamma_j}(r_j))= W_{v_m'}(\bm(a)\dot w_1 \bx_{\gamma_3}(r_3)\dots \bx_{\gamma_j}(r_j)).$$
Take an integer $k$ such that $3k\le m <4k$. By Lemma \ref{lemma21} (3), we have
$$W_{v_m}(\bm(a)\dot w_1 \bx_{\gamma_3}(r_3)\dots \bx_{\gamma_j}(r_j))=\frac{1}{\vol(U_m)}\int_{U_m} W_{v_k}(\bm(a)\dot w_1 \bx_{\gamma_3}(r_3)\dots \bx_{\gamma_j}(r_j)u)\psi_U^{-1}(u)du. $$
There is a similar formula for $W_{v_m'}$, and thus it suffices to show that
\begin{equation}\label{eq42}W_{v_k}(\bm(a)\dot w_1 \bx_{\gamma_3}(r_3)\dots \bx_{\gamma_j}(r_j)u)=W_{v_k'}(\bm(a)\dot w_1 \bx_{\gamma_3}(r_3)\dots \bx_{\gamma_j}(r_j)u), \forall u\in U_m. \end{equation}
Write $u=\bx_{\alpha}(s_4) \bx_{\gamma_3}(s_3)\bx_{\gamma_2}(s_2)\bx_{\gamma_1}(s_1)$. From the Chevalley relation, 
$$[\bx_{\alpha}, \bx_{\gamma_i}]\in \prod_{s\ge 1}U_{s\alpha+\gamma_i},$$
see \cite{St}, Chapter 3, relation (R2) above Lemma 21, we get
$$  \bx_{\gamma_3}(r_3)\dots \bx_{\gamma_j}(r_j) \bx_{\alpha}(s_4)=\bx_\alpha(s_4)\bx_{\gamma_3}(r_3')\dots \bx_{\gamma_j}(r_j'),$$
with $r_j'=r_j$. Since $N$ is abelian, we get
$$ \bx_{\gamma_3}(r_3)\dots \bx_{\gamma_j}(r_j)u =\bx_{\alpha}(s_4)\bx_{\gamma_3}(\tilde r_3)\dots \bx_{\gamma_1}(\tilde r_1),$$
with $\tilde r_k=r_k'+s_k$ for $k\ge j$. In particular, we have $|\tilde r_j|=|r_j|$ since $\bx_{\gamma_j}(r_j)\notin N_m$ and $\bx_{\gamma_j}(s_j)\in N_m$. Since $ w_1(\alpha)=\alpha$, we can write
$$\bm(a)\dot w_1 \bx_{\gamma_3}(r_3)\dots \bx_{\gamma_j}(r_j)u=\bm(a_1)\dot w_1 \bx_{\gamma_3}(\tilde r_3)\dots \bx_{\gamma_1}(\tilde r_1),$$
for some $a_1\in \GL_2(E)$.

We first show Eq.(\ref{eq42}) under the assumption $\bx_{\gamma_{j-1}}(\tilde r_{j-1})\dots \bx_{\gamma_1}(r_1)\in N_k$. Under this assumption, it suffices to show that 
$$W_{v_k}(\bm(a_1)\dot w_1 \bx_{\gamma_3}(\tilde r_3)\dots \bx_{\gamma_j}(\tilde r_j))=W_{v_k}(\bm(a_1)\dot w_1 \bx_{\gamma_3}(\tilde r_3)\dots \bx_{\gamma_j}(\tilde r_j)),$$
by Eq.(\ref{eq21}). Since $\bx_{\gamma_j}(r_j)\notin N_m$, we have $r_j\notin \CP^{-(2\textrm{ht}(\gamma_j)-1)m}$, see the structure of $H_m$ given in $\S$2, and thus $-1/r_j \in \CP^{(2\textrm{ht}(\gamma_j)-1)m}\subset \CP^{2(\textrm{ht}(\gamma_j)+1)k}$ since $m\ge 3k$. We then get $\bx_{-\gamma_j}(-1/r_j)\in J_k$. Since $|\tilde r_j|=|r_j|$, we have $\bx_{-\gamma_j}(-1/\tilde r_j)\in H_k$ too. By Eq.(\ref{eq21}), we have
$$W_{v_k}(\bm(a_1)\dot w_1 \bx_{\gamma_3}(\tilde r_3)\dots \bx_{\gamma_j}(\tilde r_j))=W_{v_k}(\bm(a_1)\dot w_1 \bx_{\gamma_3}(\tilde r_3)\dots \bx_{\gamma_j}(\tilde r_j)\bx_{-\gamma_j}(-1/\tilde r_j)).$$
There is a similar relation for $W_{v_k'}$. Thus it suffices to show 

Claim $(*)$: $ \bm(a_1)\dot w_1 \bx_{\gamma_3}(\tilde r_3)\dots \bx_{\gamma_j}(\tilde r_j)\bx_{-\gamma_j}(-1/\tilde r_j)\in BwB$ for some $w\in \bW-\wpair{w_0, w_1}$. 

We will check Claim $(*)$ case by case. We need to use the following Chevalley relation
$$ \bx_{\gamma}(r)\bx_{-\gamma}(-1/r)\in  s_\gamma B, $$
where $s_\gamma$ is the reflection defined by the root $\gamma$ (not necessarily simple), see (R3) Chapter 3 of \cite{St}.

If $j=1$, we have $\gamma_1=\beta$ and thus $\bx_{\gamma_1}(\tilde r_1)\bx_{-\gamma_1}(-1/\tilde r_1)=s_\beta b$ for some $b\in B$. Thus $$ \bm(a_1)\dot w_1 \bx_{\gamma_3}(\tilde r_3)\bx_{\gamma_2}(\tilde r_2)\bx_{\gamma_1}(\tilde r_1)\bx_{-\gamma_1}(-1/\tilde r_1)\in Mw_1U_{2\alpha+\beta}U_{\alpha+\beta}s_\beta B.$$
Since $s_\beta(2\alpha+\beta)>0, s_\beta(\alpha+\beta)>0$, we get $Mw_1U_{2\alpha+\beta}U_{\alpha+\beta}s_\beta B\subset M w_1 s_\beta B=Ms_\beta s_\alpha B $. Since $M\subset B\cup Bs_\alpha B$, we get $Ms_\beta s_\alpha B\subset Bs_\beta s_\alpha B \cup Bs_\alpha B s_\beta s_\alpha B$. Since $Bs_\alpha B s_\beta s_\alpha B=Bs_\alpha s_\beta s_\alpha B$ by Lemma 25 of \cite{St}, the assertion follows.

Next, we consider the case $j=2$. In this case $s_{\gamma_2}=s_{\alpha+\beta}=w_1$, and thus $$ \bm(a)\dot w_1 \bx_{\gamma_3}(\tilde r_3)\bx_{\gamma_2}(\tilde r_2)\bx_{-\gamma_2}(-1/\tilde r_2)\in M w_1 U_{2\alpha+\beta} w_1 B.$$
Since $w_1U_{2\alpha+\beta}w_1=U_{w_1(2\alpha+\beta)}=U_{-\beta} \subset Bs_\beta B$, we get $M w_1 U_{2\alpha+\beta} w_1 B\subset M Bs_\beta B \subset Bs_\beta B \cup Bs_\alpha Bs_\beta B=Bs_\beta B\cup Bs_\alpha s_\beta B $. The assertion follows.

Finally, we check Claim $(*)$ when $j=3$. In this case $s_{\gamma_3}=s_{2\alpha+\beta}=s_\alpha s_\beta s_\alpha=w_2$. Thus
$$\bm(a_1)\dot w_1 \bx_{\gamma_3}(\tilde r_3)\bx_{-\gamma_3}(-1/\tilde r_3)\in M w_1 w_2 B.$$
It is easy to check that $w_1w_2=s_\alpha s_\beta$ and $Ms_\alpha s_\beta B\subset Bs_\beta B\cup Bs_\alpha s_\beta B$. Thus the assertion follows.

The proof of Claim $(*)$ and hence the Proposition is finished if $\bx_{\gamma_{j-1}}(\tilde r_{j-1})\dots \bx_{\gamma_1}(r_1)\in N_k$. Otherwise, there will be an integer $i$ with $1\le i\le j-1$ such that $ \bx_{\gamma_{i-1}}(\tilde r_{i-1})\dots \bx_{\gamma_1}(r_1)\in N_k$ but $\bx_{\gamma_i}(\tilde r_i)\notin N_k$. We just repeat the above process by taking an integer $k_1$ such that $3k_1\le k<4k_1$ and then reduce everything to $W_{v_{k_1}}$. This process stops after at most 3 steps, and in each step we get a integer $k_t$ which satisfies $k_t\ge \frac{1}{4}k_{t-1}\ge\frac{1}{4^3}m\ge 4^6C$. Each step goes through since $W_{v_k}(t\dot w)=W_{v'_k}(t\dot w)$ for all $t\in T$, $w\in \bW-\wpair{w_0,w_1}$ and $k\ge 4^6C$. The proof is complete.
\end{proof}

\section{Twisting by representations of $\GL_2(E)$}

We will fix our notation as in $\S$2. In particular, $E/F$ is an unramified extension of $p$-adic fields.

\subsection{Howe vectors for Weil representations}

 Let $\mu$ be a character of $E^\times$ such that $\mu|_{F^\times}=\epsilon_{E/F}$ and hence we have the Weil representation $\omega_{\mu,\psi^{-1}}$ of $\RU(2,2)(F)$ on $\CS(E^2)$. Recall that $\psi$ is an unramified character of $F$. Given an integer $m$, we consider the function $\Phi^m\in \CS(E^2)$ defined by 
 $$\Phi^m(x,y)=\Char_{\CP^{3m}}(x)\cdot \Char_{1+\CP^{m}}(y),$$
 where for a subset $A\subset E$, $\Char_A$ denotes the characteristic function of $A$.

\begin{prop}{\label{prop51}}
\begin{enumerate}
\item For $n\in N_m$, we have
$$\omega_{\mu,\psi^{-1}}(n)\Phi^{m}=\psi_U^{-1}(n)\Phi^{m}.$$
\item For $\hat n \in \hat N_{m}$, we have $$\omega_{\mu,\psi^{-1}}(\hat n)\Phi^{m}=\Phi^{m}.$$
\end{enumerate}
\end{prop}
\begin{proof}

(1) For $n=\bn(b)\in N_m$, we have
\begin{align*}
\quad\omega_{\psi^{-1},\mu}(\bn(b))\Phi^{m}(x)
=\psi^{-1}(xb{}^t \bar x)\Phi^{m}(x).
\end{align*}
Write $x=(x_1,x_2)$ and $b=(b_{ij})$, then $$xb{}^t\! \bar x=b_{11}x_1\bar x_1+b_{21}x_1\bar x_2+b_{12}x_2\bar x_1 +b_{22}x_2\bar x_2.$$
For $(x_1,x_2)\in \supp(\Phi^{m})$, we get $x_1\in \CP_E^{3m}$ and $x_2\in 1+\CP_E^m$. From $n\in N_m$, we get $b_{11}\in \CP_F^{-5m}, b_{12}=\bar b_{21}\in \CP_E^{-3m}$ and $b_{22}\in \CP_F^{-m}$. Thus 
$$ xb{}^t\! \bar x\equiv b_{22}\mod \CO_F.$$
Thus $$\omega_{\psi^{-1},\mu}(\bn(b))\Phi^{m}=\psi^{-1}(b_{22})\Phi^{m}=\psi^{-1}_U(n)\Phi^{m,\chi}.$$

(2) For $\hat n\in \hat N_m$, we can write $\hat n= \dot w_0^{-1}\bn(b)\dot w_0$ with $$b\in \begin{pmatrix}\CP_F^{7m}& \CP_E^{5m}\\ \CP_E^{5m} & \CP_E^{3m} \end{pmatrix}\cap \Herm_2(F).$$
Let $\Phi'=\omega(w_0)\Phi^{m}$ temporarily. We have
\begin{align*}
\Phi'(x_1,x_2)&=\omega(w_0)\Phi^{m}(x_1,x_2)\\
&=\gamma_{\psi^{-1}}\int_{E^2}\Phi^m(y)\psi^{-1}(-\tr_{E/F}( x {}^t\! \bar y))dy\\
&:=\gamma_{\psi^{-1}} f_1(x_1)f_2(x_2),
\end{align*}
where
 \begin{align*}
f_1(x_1)= \int_{\CP_E^{3m}}\psi(\tr_{E/F} (\bar y_1 x_1))dy_1 =q_E^{-3m}\Char_{\CP_E^{-3m}}(x_1),
\end{align*}
 and 
\begin{align*}
f_2(x_2)=\int_{1+\CP_E^m}\psi(\tr( \bar u x_2))du =q_E^{-m}\psi(\tr(x_2))\Char_{\CP^{-m}_E}( x_2).
\end{align*}
A similar argument as above shows that $\Phi'$ is fixed by $\bn(b)$, and thus
$$\omega(\hat n)\Phi^{m}=\omega(w_0^{-1})\omega(n)\omega(w_0)\Phi^{m}=\omega(w_0^{-1})\Phi'=\omega(w_0^{-1}w_0)\Phi^{m}=\Phi^{m}.$$
This completes the proof.
\end{proof}

\begin{lem}{\label{lemma52}} Let $m$ be a positive integer.
 For $a=\begin{pmatrix}a_{11}& a_{12}\\ a_{21} &a_{22} \end{pmatrix}\in \GL_2(E)$ with $|a_{22}|\le q_E^m$ and $|a_{21}|\le q_E^{3m}$, we have
\begin{align*}&\omega_{\mu,\psi^{-1}}(\dot w_0)\Phi^{m}(e_2a) =\gamma_{\psi^{-1}} q_E^{-4m}\psi(\tr_{E/F}(a_{22})).\end{align*}
\end{lem}
\begin{proof}
We have
\begin{align*}
&\omega_{\mu,\psi^{-1}}(\dot w_0)\Phi^{m}(e_2a)\\
=&\omega_{\mu,\psi^{-1}}(\dot w_0)\Phi^{m}((a_{21},a_{22}))\\
=&\gamma_{\psi^{-1}}\int_{E^2}\Phi^m(y_1,y_2)\psi(\tr(a_{21}\bar y_1+a_{22}\bar y_2 ))dy_1dy_2\\
=&\gamma_{\psi^{-1}}\int_{\CP_E^{3m}}\psi(\tr(a_{21} \bar y_1))dy_1 \int_{1+\CP_E^m} \psi(\tr_{E/F}(a_{22}\bar y_2))dy_2\\
=&\gamma_{\psi^{-1}} q_E^{-4m}\psi(\tr_{E/F}(a_{22})).
\end{align*}
Here we used $ \psi(\tr(a_{21} \bar y_1))=1$ for $a_{21}\in \CP_E^{-3m}, y_1\in \CP_E^{3m}$ and $\psi(\tr_{E/F}(a_{22}\bar y_2))=\psi(\tr_{E/F}(a_{22})) $ for $a_{22}\in \CP_E^{-m}, y_2\in 1+\CP_E^{m}$.
\end{proof}

\subsection{Twisting by $\GL_2$}
Before we go to the proof of the local converse theorem, we recall the following result of Jacquet-Shalika.

 Let $\phi$ be a smooth complex valued function on $\GL_2(E)$ such that $ \phi(ug)=\psi_E^{-1}(u)\phi(g)$ for all $u\in N^{(2)}$, $g\in \GL_2(E)$.

\begin{prop}[Jacquet-Shalika]{\label{prop53}}
Suppose that for each integer $m$, the set $g\in \GL_2(E)$ such that $|\det(g)|=q_E^m$ and $\phi(g)\ne 0$ is contained in a compact set modulo $N^{(2)}$. Let 
$$A(\phi,W^{(2)},s)=\int_{N^{(2)}\setminus \GL_2(E)}\phi(a)W^{(2)}(a)|\det(a)|^sda,$$
and assume that this integral converges absolutely in some half plane for every irreducible representation $\tau$ and every $W^{(2)}\in \CW(\tau,\psi_E)$. If 
$A(\phi,W^{(2)},s)=0$ for all $\tau$, all $W^{(2)}$ and all $s$ when it is absolutely convergent, 
then $\phi(a)\equiv 0.$
\end{prop}
This is a corollary of Lemma 3.2 in \cite{JS}. For an argument that Lemma 3.2 in \cite{JS} implies the present form of Proposition \ref{prop51}, one can see Corollary 2.1 of \cite{Ch}, for example. \\

We are back to the notation of $\S$2. Let us repeat part of them to avoid ambiguity. We are given an unramified character $\psi$ of $F$ and and an unramified character $\psi_E$ of $E$ such that $\psi_E|_{F}$ is trivial. From them, we defined a character $\psi_U$ of $U$ such that $\psi_U|_{U_\alpha}=\psi_E^{-1}$ and $\psi_U|_{U_\beta}=\psi$. We are given two $\psi_U$-generic representations $\pi$ and $\pi'$ of $G=\RU(2,2)$ such that $\omega_\pi=\omega_{\pi'}$. For a vector $v\in V_\pi$ with $W_v(1)=1$, and an integer $m>0$, we defined the Howe vector $v_m$. Similarly, we have $v'\in V_{\pi'}$ and $v_m'$. We fixed an integer $C$ such that $v$ (resp. $v'$) is fixed by $K_C$ under $\pi$ (resp. $\pi'$). 

\begin{lem}{\label{lemma54}}
\begin{enumerate}
\item Let $a=\begin{pmatrix}a_{11} & a_{12}\\a_{21} & a_{22} \end{pmatrix}\in \GL_2(E)$. If $W_{v_m}(\bm(a)\dot w_1)\ne 0$, then $|a_{22}|_E\le q_F^{7m}=q_E^{7m/2}$ and $|a_{21}|_E\le q_F^{3m}=q_E^{3m/2}$.
\item For each $k$, the set $X_k=\wpair{a\in \GL_2(E)| |\det(a)|=q_E^k, W_{v_m}(\bm(a)\dot w_1) }$ is compact modulo $N^{(2)}$.
\end{enumerate}
\end{lem}
\begin{proof}
(1) Given $r\in \CP_{F}^{7m}$, we have $$\bm(a)w_1 \bx_{-2\alpha-\beta}(r)=\bm(a)\bx_\beta(r) w_1=\bn(b)\bm(a)w_1,$$
where $$b=\begin{pmatrix}a_{12}\bar a_{12}r & a_{12}\bar a_{22}r\\ a_{22}\bar a_{12}r & a_{22}\bar a_{22}r \end{pmatrix}.$$
Since $\bx_{-2\alpha-\beta}(r)\in J_m$, we get
$$W_{v_m}(\bm(a)w_1)=\psi(a_{22}\bar a_{22} r)W_{v_m}(\bm(a)w_1).$$
If $W_{v_m}(\bm(a)w_1)\ne 0$, we get $\psi(a_{22}\bar a_{22} r) =1$ for all $r\in \CP_F^{7m}$. Thus $a_{22} \bar a_{22}\in \CP_F^{-7m}$. Thus $|a_{22}|_E=|a_{22}\bar a_{22}|_F\le q_F^{7m}=q_E^{7m/2}$.

For the second part, we take $r\in \CP_F^{3m}$, we have the relation
$$\bm(a)w_1 \bx_{-\beta}(r)=\bm(a)\bx_{2\alpha+\beta}(r) w_1=\bn(b)\bm(a)w_1,$$
with $$b=\begin{pmatrix} a_{11}\bar a_{11}r & a_{11}\bar a_{21}r\\ a_{21}\bar a_{12} r & a_{21}\bar a_{21}r \end{pmatrix}.$$
Since $\bx_{-\beta}(r)\in J_m$, we have 
$$W_{v_m}(\bm(a)w_1)=\psi(a_{21}\bar a_{21} r) W_{v_m}(\bm(a)w_0).$$
If $W_{v_m}(\bm(a)w_1)\ne 0$, we get $\psi(a_{21} \bar a_{21} r)=1$ for all $r\in \CP_F^{3m}$, thus $a_{21}\bar a_{21}\in \CP_F^{-3m}$. Thus $|a_{21}|_E=|a_{21}\bar a_{21}|_F\le q_F^{3m}=q_E^{3m/2}$.

(2) From the Iwasawa decomposition of $\GL_2$, it suffices to show that the set $$X_k:=\wpair{\diag(a_1,a_2)\in \GL_2(F): |a_1a_2|=q_E^k, W_{v_m}(\bm(\diag(a_1,a_2))\dot w_1)\ne 0}$$ is compact. Take $r\in \CO_E$, we have $\bx_{\alpha}(-r)\in H_m$ and $\psi_U(\bx_{\alpha}(r))=1$. Write $t=\bt(a_1,a_2)$. From the relation
$$ t\dot w_1 \bx_{\alpha}(-r)=t\bx_{\alpha}(r)\dot w_1= \bx_{\alpha}(a_1r/a_2)t\dot w_1,$$
we get $\psi_U(\bx_{\alpha}(a_1r/a_2))W_{v_m}(t\dot w_1)=W_{v_m}(t\dot w_1)$.
Thus if $W_{v_m}(t\dot w_1)\ne 0$, we get $a_1/a_2\in \CO_E$. By (1), we get
$$X_k\subset \wpair{\diag(a_1,a_2): |a_1a_2|=q_F^k, a_1/a_2\in \CO_E, a_2^2\in \CP^{-7m}},$$
which is clearly compact.
\end{proof}

\begin{prop}{\label{prop55}}
Let $\mu$ be a fixed character such that $\mu|_{F^\times}=\epsilon_{E/F}$. Suppose that
 \begin{enumerate}
 \item $\gamma(s,\pi\times \mu \eta, \psi)=\gamma(s,\pi'\times \mu \eta, \psi)$ for all quasi-character $\eta$ of $E^\times$, and
 \item $\gamma(s,\pi\times \mu \tau,\psi)=\gamma(s,\pi'\times \mu \tau, \psi)$ for all irreducible representations $\tau$ of $\GL_2(E)$,\end{enumerate}
 then we have $$W_{v_m}(t\dot w)=W_{v_m'}(t\dot w)$$
for all $t\in T$, $m\ge 4^{9}C$ and $w=w_1$ or $w=w_0$.
\end{prop}

\begin{proof}
The proof is quite similar to the proof of Proposition \ref{prop41}.

Let $m= 4^9C$ and let $k=\frac{3m}{2}>m$. We have defined $\Phi^k\in \CS(E^2)$. For simplicity, we will write $\omega_{\mu,\psi^{-1}}$ as $\omega$.

For a vector $\epsilon\in V_\tau$, we take an integer $i$ such that $i\ge \max\wpair{k=3m/2, i_2(\epsilon), I( N_m, \epsilon) }$ see Lemma \ref{lemma32} and Proposition \ref{prop33} for the definition of $i_2(\epsilon)$ and $I(N_m, \epsilon)$. By Lemma \ref{lemma32}, we have defined a section $\xi_{s}^{i,\epsilon}\in I(s,\tau)$ in $\S$4.2. 

Let $W=W_{v_m}$ or $W_{v_m'}$. We will compute the integral $\Psi(W,\xi_s^{i,\epsilon}, \Phi^{k})$. This integral is defined over $U\setminus G$, and we will take the integral on the open dense subset $U\setminus NM\hat N=N^{(2)}\setminus \GL_2(E)\times \hat N$. For $g=n\bm(a)\hat n$, the Haar measure can be taken as $dg=|\det(a)|_E^{-2}dn da d\hat n. $ By the definition of $\xi_s^{i,\epsilon}$ and Eq.(\ref{eq34}), we have
\begin{align*}
&\quad \Psi(W,\xi_s^{i,\epsilon},\Phi^{k})\\
&=\int_{U\setminus G}W(g)f_{\xi_s^{i,\epsilon}}(g)(\omega(g))\Phi^{k}(e_2)dg\\
&=\int_{N^{(2)}\setminus \GL_2(E)\times \hat N}W(\bm(a)\hat n)f_{\xi_s^{i,\epsilon}}(\bm(a)\hat n)\omega(\bm(a)\hat n)\Phi^{k}(e_2)|\det(a)|^{-2}da d\hat n\\
&=\int_{N^{(2)}\setminus \GL_2(E)\times \hat N_i}W(\bm(a)\hat n)|\det(a)|^{s-3/2} W_{\epsilon}^{(2)}(a)\omega(\bm(a)\hat n)\Phi^{k}(e_2)da d\hat n.
\end{align*}
Since $i\ge k>m$, we have $ \hat N_i\subset \hat N_k\subset \hat N_m$. Thus for $\hat n \in  \hat N_i$, we have $W(g\hat n)=W(g)$ and $\omega(\hat n)\Phi^{k}=\Phi^{k}$ by Lemma \ref{lemma21} and Proposition \ref{prop51}. Thus we have
\begin{align}
&\quad \Psi(W,\xi_s^{i,\epsilon},\Phi^{k}) \label{eq58}\\
&=\vol(\hat N_i)\int_{E^1N^{(2)}\setminus \GL_2(E)} W(\bm(a))W_{\epsilon}^{(2)}(a)|\det(a)|^{s-1}\mu(\det(a))\Phi^{k}(e_2a)da,\nonumber
\end{align}
where $E^1$ embeds into $T^{(2)}\subset \GL_2(E)$ diagonally.

Since $M\subset B\cup Bs_\alpha B$, and we have showed that $W_{v_m}(g)=W_{v_m'}(g)$ for $g\in B\cup Bs_\alpha B$ in Corollary \ref{cor26},  we get
\begin{equation}
\Psi(W_{v_m},\xi_s^{i,\epsilon}, \Phi^{k})=\Psi(W_{v_m'}, \xi_s^{i,\epsilon},\Phi^{k}).\label{eq59}
\end{equation}

Next, we compute the integral $\Psi(W,\tilde \xi^{i,\epsilon}_{1-s}, \Phi^{k})$ for $\tilde \xi^{i,\epsilon}_{1-s}=M(s)(\xi_{s}^{i,\epsilon})$. We will take this integral on the open dense set $U\setminus NMw_1N$ of $U\setminus G$. 
\begin{align}
&\quad \Psi(W,\tilde\xi_{1-s}^{i,\epsilon},\Phi^{k}) \nonumber\\
&=\int_{N^{(2)}\setminus \GL_2(E)\times N} W(\bm(a)w_1n) f_{\tilde \xi_{1-s}^{i,\epsilon}}(\bm(a)w_1n)\omega(\bm(a)w_1n)\Phi^{k}(e_2)|\det(a)|^{-2}da dn\nonumber\\
\label{eq510}&=\int_{N^{(2)}\setminus \GL_2(E)\times N_m}W(\bm(a)w_1n) f_{\tilde \xi_{1-s}^{i,\epsilon}}(\bm(a)w_1n)\omega(\bm(a)w_1n)\Phi^{k}(e_2)|\det(a)|^{-2}dadn\\
\label{eq511}&+\int_{N^{(2)}\setminus \GL_2(E)\times (N-N_m)}W(\bm(a)w_1n) f_{\tilde \xi_{1-s}^{i,\epsilon}}(\bm(a)w_1n)\omega(\bm(a)w_1n)\Phi^{k}(e_2)|\det(a)|^{-2}dadn.
\end{align}
Since $i\ge I(N_m,\epsilon)$, we have 
$$f_{\tilde \xi_{1-s}^{i,\epsilon}}( \bm(a)w_1n)=\vol(\hat N_i)|\det(a)|^{3/2-s}\tilde W^{(2)}_{\epsilon}(a), i\ge I(N_m, \epsilon), n\in N_m,$$
see Proposition \ref{prop33} and Eq.(\ref{eq35}).

For $n\in N_m\subset N_k$, by Lemma \ref{lemma21} and Proposition \ref{prop51} (2), we have $$W(\bm(a)w_1n)=\psi_U(n)W(\bm(a)w_1), \omega(\bm(a)w_1n)\Phi^{k}(e_2)=\psi_U^{-1}(n)\omega(\bm(a)w_1)\Phi^{k}(e_2).$$ 
Thus the term Eq.(\ref{eq510}) becomes
\begin{align*}
\vol(\hat N_i)\vol(N_m)\int_{N^{(2)}\setminus \GL_2(E)} W(\bm(a)w_1) \tilde W_{\epsilon}(a)|\det(a)|^{-s-1/2}\omega(\bm(a)w_1)\Phi^{k,\chi}(e_2)da.
\end{align*}

By assumption (1) and Proposition \ref{prop42}, we get $W_{v_m}(\bm(a)\dot w_1n)=W_{v_m'}(\bm(a)\dot w_1n)$ for $n\in N- N_m$. Thus the term Eq.(\ref{eq511}) for $W=W_{v_m}$ and for $W=W_{v_m'}$ are the same. We then get
\begin{align}
&\Psi(W_{v_m}, \tilde \xi_{1-s}^{i,\epsilon},\Phi^{k,\chi})-\Psi(W_{v_m'}, \tilde \xi_{1-s}^{i,\epsilon},\Phi^{k}) \label{eq512}\\
=&\vol(\hat N_i)\vol(N_m)\int_{E^1N^{(2)}\setminus \GL_2(E)} (W_{v_m}(\bm(a)w_1)-W_{v_m'}(\bm(a)w_1)) \nonumber\\
\quad &\cdot \omega(\bm(a)w_1)\Phi^{k}(e_2) \tilde W_{\epsilon}(a)|\det(a)|^{-s-1/2}da.\nonumber
\end{align}
By Eq.(\ref{eq59}), Eq.(\ref{eq512}) and the local functional equation, we get
\begin{align}
&d_m ( \gamma(s,\pi, \omega_{\mu,\psi^{-1},\chi},\tau)-\gamma(s,\pi', \omega_{\mu,\psi^{-1},\chi},\tau))\label{eq513}\\
\nonumber=&\int_{N^{(2)}\setminus \GL_2(E)} (W_{v_m}(\bm(a)w_1)-W_{v_m'}(\bm(a)w_1)) \omega(\bm(a)w_1)\Phi^{k}(e_2) \tilde W_{\epsilon}(a)|\det(a)|^{-s-1/2}da,
\end{align}
where $d_m=\Psi(W, \xi_s^{i,m}, \Phi^{k}) \vol(N_m)^{-1}\vol(\hat N_i)^{-1}$, which is independent of $i$.

By our assumption on $\gamma$-factors, we have
$$\int_{N^{(2)}\setminus \GL_2(E)} (W_{v_m}(\bm(a)w_1)-W_{v_m'}(\bm(a)w_1)) \omega(\bm(a)w_1)\Phi^{k}(e_2) \tilde W_{\epsilon}(a)|\det(a)|^{-s-1/2}da\equiv 0. $$
By Proposition \ref{prop53} and Lemma \ref{lemma54} (2), we get

\begin{equation}\label{eq514}
(W_{v_m}(\bm(a)w_1)-W_{v_m'}(\bm(a)w_1)) \omega(\bm(a)w_1)\Phi^{k}(e_2)\equiv 0.
\end{equation}

For $t=\bt(a_1,a_2)\in T$, we have $$\bt(a_1,a_2)\dot w_1=\bm\begin{pmatrix}& a_1\\ a_2 & \end{pmatrix}\dot w_0.$$
By Lemma \ref{lemma52}, for $|a_2|\le q_E^{3k}$, we have $\omega(t\dot w_1)\Phi^{k}(e_2)\ne 0$. By Eq.(\ref{eq514}), we get
\begin{equation} \label{eq515}W_{v_m}(t\dot w_1)=W_{v_m'}(t\dot w_1), \forall t=\bt(a_1,a_2) \textrm{ with }|a_2|\le q_E^{3k}.\end{equation}
From Lemma \ref{lemma54} (1),  we get \begin{equation}\label{eq516}W_{v_m}(t\dot w_1)=W_{v_m'}(t\dot w_1)=0, \textrm{ if }  |a_{2}|> q_E^{7m/2}.\end{equation}  Since $3k>7m/2$, by Eq.(\ref{eq515}) and Eq.(\ref{eq516}), we get 
$$W_{v_m}(t\dot w_1)=W_{v_m'}(t\dot w_1), \forall t\in T.$$

On the other hand, we have 
$$\bt(a_1,a_2)\dot w_0=\bm (a')\dot w_1, \textrm{ with } a'=\begin{pmatrix}& a_1\\ a_2 & \end{pmatrix}\in \GL_2(E).$$
By Lemma \ref{lemma52} , for $|a_2|\le q_E^k$, we have
$$\omega(\bt(a_1,a_2)\dot w_0)\Phi^{k}(e_2)\ne 0.$$
By Eq.(\ref{eq514}), we get
\begin{equation}\label{eq517}W_{v_m}(\bt(a_1,a_2)w_0) =W_{v_m}(\bm(a')w_1)=W_{v_m'}(\bm(a')w_1)=W_{v_m'}(\bt(a_1,a_2)w_0), \textrm{ if } |a_2|\le q_E^k.\end{equation}
By Lemma \ref{lemma54} (1), we have
\begin{equation}\label{eq518}W_{v_m}(\bt(a_1,a_2)w_0)=0=W_{v_m'}(\bt(a_1,a_2)w_0), \textrm{ if } |a_2|_E> q_E^{3m/2}=q_E^k. \end{equation}
From Eq.(\ref{eq517}) and Eq.(\ref{eq518}), we get 
$$W_{v_m}(tw_0)=W_{v_m'}(tw_0), \forall t\in T.$$
Thus we proved $W_{v_m}(tw_1)=W_{v_m'}(tw_1)$ and $W_{v_m}(tw_0)=W_{v_m'}(tw_0)$ for all $t\in T$ and $m=4^9L$. The same is true for $m\ge 4^9L$ by Lemma \ref{lemma21} (3) and Proposition \ref{prop25}. This finishes the proof.
\end{proof}
\begin{thm}[Local Converse Theorem for $\RU_{E/F}(2,2)$ when $E/F$ is unramified] \label{thm56}
Assume $E/F$ is unramified. Let $\pi,\pi'$ be two $\psi_U$-generic irreducible smooth representations of $\RU_{E/F}(2,2)$ with the same central character. If $$\gamma(s,\pi\times  \eta,\psi_U)=\gamma(s,\pi'\times \eta,,\psi_U), \textrm{ and } \gamma(s,\pi\times \tau,\psi_U)=\gamma(s,\pi'\times \tau,\psi_U),$$
for all quasi-character $\eta$ of $E^\times$ and all irreducible smooth representation $\tau$ of $\GL_2(E)$, then $\pi\cong \pi'$.
\end{thm}
\begin{proof}
If $\psi_U$ is unramified, by Proposition \ref{prop55} and Proposition \ref{prop25}, we get 
$$W_{v_m}(g)=W_{v'_m}(g),$$
for all $g\in G$. Thus $\pi\cong \pi'$ by the uniqueness of the Whittaker functional. If $\psi_U$ is not unramified, it suffices to modify the above proof a little bit.
\end{proof}

\section{A new local zeta integral for $\RU(2,2)\times \Res_{E/F}(\GL(1))$}
The local zeta integral for $\RU(2,2)\times \Res_{E/F}(\GL_1)$ we considered in $\S$1 is the analogue of the $\Sp(n)\times \GL_m$ case developed by Ginzburg, Rallis and Soudry in \cite{GRS1, GRS2}, which comes from a global zeta integral. 

In this section, we consider a new local zeta integral for $\RU(2,2)\times \Res_{E/F}(\GL_1)$ in the generic case and prove the local functional equation and hence the existence of the $\gamma$-factors. We also prove that this new gamma factor can be used to obtain the local converse theorem. The advantage of the new local zeta integral is that it does not involve the Weil representation so that it is simpler than the old one. But it is not clear if the new local zeta integral comes from a global zeta integral.

\subsection{A new local zeta integral for $\RU(2,2)\times \Res_{E/F}(\GL_1)$ and the local functional equation}
Let $F$ be a $p$-adic field, $E/F$ be a quadratic field extension. Let $G=\RU(2,2)(F)$ and $\pi$ be an irreducible admissible smooth representation of $G$. Let $\psi$ (resp. $\psi_E$) be a fixed nontrivial additive character of $F$ (resp. $E$). We require that $\psi_E|_F$ is trivial. We define a generic character $\psi_U$ on $U$ as before, i.e., $\psi_U|_{U_\alpha}=\psi_E$ and $\psi_U|_{U_\beta}=\psi$. 

We assume $\pi$ is $\psi_U$-generic. Let $\CW(\pi,\psi)$ be the set of Whittaker functions of $\pi$. For $W\in \CW(\pi,\psi)$, and a quasi-character $\eta$ of $E^\times$, we consider the integral
 $$\Psi(s,W,\eta)=\int_{E^\times}W(\bt(a))\eta(a)|a|_E^{s-3/2}da.$$
 Recall that $\bt(a)=\bm(\diag(a,1)).$
From a standard estimate, we can show that $\Psi(s,W,\eta)$ is absolutely convergent when $\Re(s)>>0$ and defines a rational function of $q_E^{-s}$.

We consider the integral
\begin{align}&\quad \widetilde \Psi(1-s, W, \eta^{*}) \nonumber\\
&=\int_{E^\times} \int_{E\times F}W\left(\bt(a)\dot w_2 \bx_{\alpha+\beta}(x)\bx_{2\alpha+\beta}(y)\right)dxdy\eta^*( a)|a|^{-s-1/2}da, \label{eq63}
\end{align}
where $\eta^*(a)=\eta(\bar a^{-1})$. Similarly, one can show that $\widetilde \Psi(1-s,\widetilde W, \eta^{-1})$ is absolutely convergent for $\Re(s)<<0$ and defines a meromorphic function of $q_E^{-s}$.

Consider the linear form $W\mapsto B(W)=\widetilde \Psi(1-s,  W, \eta^{*})$ on $\CW( \pi, \psi_U)$. 
\begin{lem}\label{lem63} Let $\psi_N$ be the character of $N$ defined by $\psi_N=\psi_U|_{N}$. Then we have
$$B(\pi(n)W)=\psi_N(n)B(W), n\in N,$$
and 
$$B(\pi(\bt(a))W)=\eta^{-1}(a)|a|^{3/2-s}B(W), a\in E^\times.$$
\end{lem}
\begin{proof}
This follows from a straightforward matrix calculation, and we omit the details.
\end{proof}

\begin{prop}\label{prop64}
Except for finite number of $q^{-s}$, up to a scaler there is at most one linear functional $A$ on $\CW(\pi,\psi_U)$ such that
$$A(\pi(n)W)=\psi_N(n)A(W)$$
and 
$$A(\pi(\bt(a))W)=\eta^{-1}(a)|a|_E^{-s+3/2}.$$
\end{prop}
\begin{proof}
The first condition says that $A\in \Hom_\BC(\pi_{N.\psi_N},\BC)$. The twisted Jacquet module $\pi_{N,\psi_N}$ defines a representation of the mirabolic subgroup $P_2^{(2)}\subset \GL_2\cong M$, where 
$$P_2^{(2)}=\wpair{\begin{pmatrix}a& x\\ &1 \end{pmatrix}}\subset \GL_2(E)$$
 The second condition says that 
$$A\in \Hom_{T^{(2)}}(\pi_{N,\psi_N},\eta^{-1}|~|^{-s+3/2}),$$
where $T^{(2)}=\wpair{\diag(a,1)}\subset \GL_2(E)$. As before, let $N^{(2)}$ be the unipotent subgroup of $P_2^{(2)}$. As a representation of $P_2^{(2)}$, we have the exact sequence
$$0\ra \ind_{N^{(2)}}^{P_2^{(2)}}((\pi_{N,\psi_N})_{N^{(2)},\psi_E})\ra \pi_{N,\psi_N}\ra (\pi_{N,\psi_N})_{N^{(2)}}\ra 0.$$
Note that $ (\pi_{N,\psi_N})_{N_2,\psi_E}\cong \pi_{U,\psi_U}$ has dimension 1 by the uniqueness of Whitaker functionals, and $ (\pi_{N,\psi_N})_{N_2}$ has finite dimension by the following proposition, Proposition \ref{prop65}. The above sequence is equivalent to
\begin{equation}\label{eq64}0\ra \ind_{N^{(2)}}^{P_2^{(2)}}(\psi_E)\ra  \pi_{N,\psi_N}\ra (\pi_{N,\psi_N})_{N^{(2)}}\ra 0.\end{equation}
In the proof of the local functional equation for local zeta integral of $\GL_2(E)$ in \cite{JL}, Jacquet-Langlands showed that 
$$\Hom_{T^{(2)}}(\ind_{N^{(2)}}^{P_2^{(2)}}(\psi),\eta^{-1}|~|^{-s+3/2})$$
has dimension 1. Since $  (\pi_{N,\psi_N})_{N^{(2)}}$ has finite dimension, after excluding finite number of $q_E^{-s}$, we have $\dim \Hom_{T_2}(\pi_{N,\psi_N},\eta^{-1}|~|^{-s+3/2})\le 1$ by the exact sequence (\ref{eq64}).
\end{proof}
\begin{prop}[Kazhdan]\label{prop65}
Let $(\pi,V)$ be an irreducible smooth representation of $G$, and $\theta$ be the character on $U$ defined by 
$$\theta\left(\begin{pmatrix}1&x&&\\ &1&&\\ &&1&\\ &&-\bar x&1 \end{pmatrix}\begin{pmatrix}1&&b_{11}&b_{12}\\ &1& \bar b_{12}& b_{22}\\ &&1&\\ &&&1 \end{pmatrix}\right)=\psi(b_{22}).$$
Then the twisted Jacquet module $V_{U,\theta}$ has finite dimension.
\end{prop}
In the $\GL_n$ case, this is a Theorem of Kazhdan, \cite{BZ} Theorem 5.21. The proof in our case is similar and we omit the details.

\begin{cor}\label{cor66}
There is a meromorphic function $\gamma'(s,\pi,\eta)$ such that
$$\widetilde \Psi(1-s,W,\eta^{*})=\gamma'(s,\pi\times\eta,\psi_U)\Psi(s,W,\eta), \forall W\in \CW(\pi,\psi_U).$$
\end{cor}
\begin{proof}
Consider the linear functional $W\mapsto A(W)=\Psi(s,W,\eta)$ on $\CW(\pi,\psi_U)$. It is clear that 
$$A(\pi(n)W)=\psi_N(n)A(W), n\in N,$$
and 
$$A(\pi(\bt(a))W)=\eta^{-1}(a)|a|^{3/2-s}A(W).$$
Now assertion follows from Lemma \ref{lem63} and Proposition \ref{prop64}.
\end{proof}

\subsection{Proof of the local converse theorem using the new $\gamma$-factor}
Now we are going to prove the local converse theorem using the new gamma factor $\gamma'(s,\pi\times\eta,\psi)$. We recall our notations in $\S$2. We assume $E/F$ is unramified and we are given two $\psi_U$-generic representation $\pi, \pi'$ of $G=\RU(2,2)(F)$ with the same central character. For $v\in V_\pi$, (resp $ v'\in V_{\pi'}$) with $W_v(1)$ (resp. $W_{v'}(1)=1$), we have defined Howe vectors $v_m$ (resp. $v'_m$) for positive integers $m$. Let $C$ be an integer such that $v$ and $v'$ are both fixed by $K_C$. 
\begin{prop}{\label{thm67}}
 If $\gamma'(s,\pi\times\eta,\psi_U)=\gamma'(s,\pi'\times\eta,\psi_U)$ for all quasi-characters $\eta$ of $F^\times$, then 
$$W_{v_m}(tw)=W_{v_m'}(tw)$$
for $t\in T$, and all $m\ge 4^6C$. 
\end{prop}
\begin{proof}
Let $m\ge 4^6C$. By Corollary \ref{cor23}, we have $W_{v_m}(\bt(a))=0$ for $a\notin 1+\CP_E^m$. By Lemma \ref{lemma21}, we have $W_{v_m}(\bt(a))=1$ for $a\in 1+\CP_E^m$. Thus we have
 \begin{align*}\Psi(s,W_{v_m},\eta)&=\int_{E^\times}W_{v_m}(t(a))\eta(a)|a|_E^{s-3/2}da\\
&=\int_{1+\CP_E^m}\eta(a)da.
\end{align*}
The same calculation works for $W_{v_m'}$. Thus we have
\begin{equation}\label{eq65} \Psi(s,W_{v_m},\chi)=\Psi(s,W_{v_m'},\chi). \end{equation}

Let $W=W_{v_m}$ or $W_{v_m'}$. We have
\begin{align}
&\quad\widetilde \Psi(1-s, W,\eta^{*}) \nonumber\\
&=\int_{E^\times}\int_{E\times F}W(\bt( a)w \bx_{\alpha+\beta}(x)\bx_{2\alpha+\beta}(y))dxdy \eta^{*}(a)|a|^{-s-1/2}da \nonumber \\
&=\int_{E^\times}\int_{x\in \CP_E^{-3m},y\in \CP_F^{-5m}}W(\bt( a)\dot w_2 \bx_{\alpha+\beta}(x)\bx_{2\alpha+\beta}(y))dxdy \eta^{*}(a)|a|^{-s-1/2}da \label{eq66} \\
&+\int_{E^\times}\int_{x\notin \CP_E^{-3m}, \textrm{ or }y\notin \CP_F^{-5m}}W(\bt( a) \dot w_2\bx_{\alpha+\beta}(x)\bx_{2\alpha+\beta}(y))dxdy \eta^{*}(a)|a|^{-s-1/2}da.\label{eq67}
\end{align}
 For $x\in \CP_E^{-3m}, y\in \CP_F^{-5m}$, we have $\bx_{\alpha+\beta}(x)\bx_{2\alpha+\beta}(y) \in H_m$, and thus 
$$W(\bt(a)\dot w_2 \bx_{\alpha+\beta}(x)\bx_{2\alpha+\beta}(y))=W(\bt(a)\dot w_2),$$
by Lemma \ref{lemma21} or Eq.(\ref{eq21}) .  Then the term Eq.(\ref{eq66}) becomes 
$$\vol(\CP_E^{-3m})\vol(\CP_F^{-5m})\int_{E^\times} W(\bt( a)w)\eta^{*}(a)|a|^{-s-1/2}da.$$
By Corollary \ref{cor26}, we have $W_{v_m}(\bt( a)w \bx_{\alpha+\beta}(x)\bx_{2\alpha+\beta}(y) )=W_{v_m'}(\bt(a)w\bx_{\alpha+\beta}(x)\bx_{2\alpha+\beta}(y))$ for $x\notin \CP_E^{-3m}$ or $y\notin \CP_F^{-5m}$. Thus the term Eq.(\ref{eq67}) for $W=W_{v_m}$ and for $W=W_{v_m'}$ are the same. 
Thus
\begin{align}
&\widetilde \Psi(1-s, W_{v_m}, \eta^{*})-\tilde \Psi(1-s, W_{v_m'}, \eta^{*})\label{eq68}\\
=& \vol(\CP_E^{-3m})\vol(\CP_F^{-5m})\int_{E^\times}( W_{v_m}(\bt( a)\dot w)-W_{v_m'}(\bt( a)\dot w))\eta^{*}(a)|a|^{-s-1/2}da. \nonumber
\end{align}
By Eq.(\ref{eq65}), Eq.(\ref{eq68}) and the local functional equation, i.e., Corollary \ref{cor66}, we have
\begin{align}
&d_m(\gamma'(s,\pi\times \eta,\psi_U)-\gamma'(s,\pi'\times\eta,\psi_U))\label{eq69}\\
=&\int_{E^\times}( W_{v_m}(\bt( a)w)-W_{v_m'}(\bt( a)w))\eta^{*}(a)|a|^{-s-1/2}da, \nonumber
\end{align}
where $d_m=\Psi(s,W_{v_m}, \eta)$. Thus by the assumption, we get
$$\int_{E^\times}( W_{v_m}(\bt( a)\dot w_2)-W_{v_m'}(\bt( a)\dot w_2))\eta^{*}(a)|a|^{-s-1/2}da\equiv 0,$$
which implies that $W_{v_m}(\bt(a)\dot w_2)=W_{v_m'}(\bt(a)\dot w_2)$. The following proof is the same as the proof of Proposition \ref{prop41}.
\end{proof}
From the proof of Theorem \ref{thm56}, we see that it can be restated as
\begin{thm}[Local Converse Theorem for generic representations of unramified $\RU(2,2)$]
Suppose that $E/F$ is unramified. Let $\pi,\pi'$ be two $\psi_U$-generic irreducible representations of $G$ with the same central character. Suppose that
$$\gamma'(s,\pi\times\eta,\psi_U)=\gamma'(s,\pi'\times\eta,\psi_U), \gamma(s,\pi\times \tau,\psi_U)=\gamma(s,\pi\times\tau,\psi_U) $$
for all quasi-characters $\eta$ of $E^\times$ and all irreducible representations $\tau$ of $\GL_2(E)$,
then $\pi\cong \pi'$.
\end{thm}

\section{Concluding remarks}
\subsection{When $E/F$ is ramified}
In this subsection, we give a brief account on how to modify the method we used in previous sections for unramified $E/F$ so that it adapts when $E/F$ is ramified and the residue characteristic of $F$ is odd.

We assume that the field extension $E/F$ is ramified. Let $p$ be the characteristic of the residue field of $F$. Let $\CP_E$ (resp. $\CP_F$) be the maximal ideal of $\CO_E$ (resp. $\CO_F$), and let $p_E$ (resp. $p_F$) be a prime element of $E$ (resp, $F$). Then we have $\CP_F\cdot \CO_E= \CP_E^2$ and $p_F=p_E^2 u$ for some $u\in \CO_E^\times$.

Let $\fD_{E/F}=\CP_E^d$ be the different of $E/F$ for some positive integer $d$.

Let $x\in \CO_E$ be an element such that $\CO_E=\CO_F[x]$. Let $v_E$ be the discrete valuation of $E$. Define $i_{E/F}=v_E(\bar x-x)$. The integer $i_{E/F}$ is independent of choice of $E/F$. It is known that $i_{E/F}=1$ if $p\ne 2$, and $i_{E/F}\ge 2$ if $p=2$, see Chapter IV of \cite{Se}.

Let $\psi_U$ be the unramified generic character of $U$ as in $\S$2. In the ramified case, we need to change the definition of Howe vectors a little bit. The reason is that $\CP_E\cap F\ne \CP_F$. To remedy this, we basically need to replace $\CP_E$ in $\S$3 by $\CP_F\cdot \CO_E=\CP_E^2$. We give some details on this. We define $J_m=d_m K_{2m}d_m^{-1}$, where $d_m=t(p_F^{-3m}, p_F^{-m})$ and $K_{2m}=(1+\textrm{Mat}_{4\times 4}(\CP_E^{2m}))$ are defined in $\S$3. Define the character $\psi_m$ on $J_m$ in the same way as in $\S$3. We still have $\psi_m|_{U_m}=\psi_U|_{U_m}$. 

Let $(\pi,V_\pi)$ be a $\psi_U$-generic representation of $G=\RU(2,2)$. Starting from a vector $v\in V_\pi$, we define the Howe vectors $v_m$ in the same way. Let $C$ be an integer such that $v$ is fixed by $K_{2C}$, then Lemma \ref{lemma21} still holds. Part (1) of Lemma \ref{lemma22} should be modified to 
\begin{lem}{\label{lem71}}
Let $t=\bt(a_1,a_2)\in T$. If $W_{v_m}(t)\ne 0$, then $a_1/a_2\in 1+\CP_E^{2m}$ and $a_2\bar a_2\in 1+\CP_F^m$.
\end{lem}

Now assume that $(\pi,V_{\pi})$ and $(\pi',V_{\pi'})$ are two $\psi_U$-generic representations of $G=\RU_{E/F}(2,2)$ with the same central character as in $\S$2. We take $v\in V_\pi, v'\in V_{\pi'}$ and define Howe vectors $v_m, v_m'$. Let $C$ be an integer such that $v,v'$ are fixed by $K_{2C}$.
\begin{lem}
If the residue characteristic $p$ of $F$ is not $2$, then $W_{v_m}(t)=W_{v_m'}(t)$ for all $t\in T$.
\end{lem}
\begin{proof}
Let $t=\bt(a_1,a_2)\in T$ with $a_1,a_2\in E^\times$. If $W_{v_m}(t)\ne 0$, then $a_1/a_2\in 1+\CP_E^{2m}$ and $a_2\bar a_2\in 1+\CP_F^m$ by Lemma \ref{lem71}. If $p\ne 2$, we have $\Nm_{E/F}(1+\CP_E^{2m})=1+\CP_F^m$ by Corollary 3, $\S$3, Chapter V of \cite{Se}. From this, we get $a_2\in E^1(1+\CP_E^{2m})$. Notice that $\bt(a_1,a_2)\in H_m$ for $a_1,a_2\in 1+\CP_E^{2m}$, the following proof is the same as in the unramified case. 
\end{proof}
\begin{rem}\upshape
If $p=2$, then Corollary 3, $\S$3, Chapter V of \cite{Se} says that $\Nm_{E/F}(1+\CP_E^{2m-i_{E/F}+1})=1+\CP_E^F$ for $m\ge i_{E/F}$. Thus from $a_2\bar a_2\in 1+\CP_F^m$, we only get $a_2\in E^1(1+\CP_E^{2m-i_{E/F}+1})$ for $m\ge i_{E/F}$. Note that $T\cap H_m=\bt(1+\CP_E^{2m},1+\CP_E^{2m})$ and thus $\bt(a_1, a_2)\notin J_m$ for some $a_2\in 1+\CP_E^{2m-i_{E/F}+1}$ and $a_1$ with $a_1/a_2\in 1+\CP_E^{2m}$, since $i_{E/F}>1$ in the case $p=2$. Thus we cannot conclude that $W_{v_m}(\bt (a_1, a_2))=W_{v_m'}(\bt(a_1,a_2))$ from Lemma \ref{lemma21} (2) or Eq.(\ref{eq21}). This is the reason that we need to exclude the case $p=2$.
\end{rem}

 With slightly modification, one can check easily that the proof of the local converse theorem goes through. We omit the details.
 
 \subsection{Local converse theorem for $\Sp_4(F)$ and $\widetilde \Sp_4(F)$}
The gamma factors for generic irreducible smooth representations (resp. genuine generic irreducible smooth representations) of $\Sp_4(F)$ (resp. $\widetilde \Sp_4(F)$) are studied in \cite{Ka}. The local zeta integrals in these cases are defined in the same manner as the $\RU_{E/F}(2,2)$ case as we considered before. Thus, with similar argument, we can obtain the local converse theorem in these cases, i.e., we have
\begin{thm}
Let $F$ be a $p$-adic local field with odd residue characteristic. Let $\pi,\pi'$ be two $\psi_U$-generic irreducible smooth representations $($resp. genuine $\psi_U$-generic irreducible smooth representations of $)$ $\Sp_4(F)$ (resp. $\widetilde \Sp_4(F)$) with the same central character. If $$\gamma(s,\pi\times \eta,\psi)=\gamma(s,\pi'\times \eta,\psi), \textrm{ and } \gamma(s,\pi\times \tau, \psi)=\gamma(s,\pi'\times \tau,\psi),$$
for all quasi-characters $\eta$ of $F^\times$ and all irreducible smooth representations of $\GL_2(F)$, then $\pi\cong \pi'$.
\end{thm}
\begin{rem}\upshape We remark that, using the local Langlands conjecture for $\Sp_4(F)$ which is now known by the work of Gan-Takeda \cite{GT1,GT2}, and the recently proved Jacquet's conjecture for local converse theorem for $\GL_5$ \cite{ALSX, JLiu, Chai}, the local converse theorem for $\Sp_4(F)$ is now known without any restriction on $F$ and the central character. In fact, given an arbitrary $p$-adic field $F$, suppose that we have two $\psi_U$-generic irreducible representations $\pi,\pi'$ of $\Sp_4(F)$ such that 
$$\gamma(s,\pi\times \eta,\psi)=\gamma(s,\pi'\times \eta,\psi), \textrm{ and } \gamma(s,\pi\times \tau, \psi)=\gamma(s,\pi'\times \tau,\psi),$$
for all quasi-characters $\eta$ of $F^\times$ and all irreducible smooth representations $\tau$ of $\GL_2(F)$.
Using the local Langlands correspondence for $\Sp_4$, we get two Langlands parameters $\phi,\phi': WD(F)\ra \SO_5(\BC)\subset \GL_5(\BC)$. Now apply the local Langlands correspondence for $\GL_n$ \cite{HT, He}, we get two irreducible representations $\sigma, \sigma'$ of $\GL_5(\BC)$. In each step $\pi\mapsto \phi\mapsto \sigma$, the gamma factors are preserved. Thus we get
$$\gamma(s,\sigma\times \eta,\psi)=\gamma(s,\sigma'\times \eta,\psi), \textrm{ and } \gamma(s,\sigma\times \tau, \psi)=\gamma(s,\sigma'\times \tau,\psi),$$
for all quasi-characters $\eta$ of $F^\times$ and all irreducible smooth representations $\tau$ of $\GL_2(F)$. From the Jacquet's conjecture for the local converse problem for $p$-adic $\GL_n$, which is recently proved for prime $n$ in \cite{ALSX} and for general $n$ in \cite{JLiu, Chai}, we get $\sigma\cong \sigma'$. Since the local Langlands correspondence for $\GL_n$ is bijective, we get $\phi=\phi'$ (up to equivalence). Thus we get $\pi$ and $\pi'$ are in the same $L$-packet, say $\Pi_\varphi$. In \cite{GT2}, it is shown that in each $L$-packet, there is at most one $\psi_U$-generic representation. Thus we get $\pi\cong \pi'$.
\end{rem}

\end{document}